\documentclass[11pt,reqno]{amsart}
\usepackage[T1]{fontenc}
\usepackage[utf8]{inputenc}
\usepackage{lmodern,amsmath,amssymb,mathtools,mathrsfs}
\usepackage{microtype,needspace}
\usepackage[a4paper,margin=28mm,headheight=14pt]{geometry}
\usepackage{booktabs,array}
\usepackage[hidelinks,bookmarksnumbered]{hyperref}
\hypersetup{pdftitle={Positive cubature on S²: low-degree rigidity and uniform bounds},pdfauthor={Zhuo Cheng and Deyu Yu},pdfsubject={Positive cubature, node bounds, and radial cap methods},pdfkeywords={positive cubature, spherical design, Toeplitz completion, effective bounds}}
\numberwithin{equation}{section}
\newtheorem{theorem}{Theorem}[section]
\newtheorem{proposition}[theorem]{Proposition}
\newtheorem{lemma}[theorem]{Lemma}
\newtheorem{corollary}[theorem]{Corollary}
\newtheorem{conjecture}[theorem]{Conjecture}
\theoremstyle{definition}

\theoremstyle{remark}

\newcommand{\R}{\mathbb R}
\newcommand{\C}{\mathbb C}
\newcommand{\Sph}{S^2}
\newcommand{\RP}{\mathbb {RP}^2}
\newcommand{\dd}{\,d}
\newcommand{\Acal}{\mathcal A}
\newcommand{\Kcal}{\mathcal K}
\newcommand{\eps}{\varepsilon}
\newcommand{\ind}{\mathbf 1}
\DeclareMathOperator{\supp}{supp}
\DeclareMathOperator{\rank}{rank}

\DeclareMathOperator{\odd}{odd}
\DeclareMathOperator{\osc}{osc}
\DeclareMathOperator{\sgn}{sgn}
\DeclareMathOperator{\tr}{tr}
\DeclareMathOperator{\Cov}{Cov}
\newcommand{\norm}[1]{\lVert#1\rVert}

\allowdisplaybreaks[1]
\title[Positive cubature on $S^2$]{Positive cubature on $S^2$:\\low-degree rigidity and uniform bounds}
\author{Zhuo Cheng}
\address{School of Mathematics and Statistics, HNP-LAMA, Central South University, Changsha, Hunan 410083, P. R. China}
\email{zhuo@outlook.cz}

\author{Deyu Yu}
\address{School of Mathematics and Statistics, HNP-LAMA, Central South University, Changsha, Hunan 410083, P. R. China}
\email{242107002@csu.edu.cn}

\date{30 September 2026}
\subjclass[2020]{Primary 65D32; Secondary 05B30, 42C10, 46E35, 52C17}
\keywords{positive cubature, spherical design, hemisphere rank, Toeplitz moments, linear programming, spherical packing, cap variational problem}
\begin{document}
\begin{abstract}
Let \(N_t\) denote the least number of nodes in a positive cubature formula of degree \(t\) on \(S^2\). We prove that a formula of degree \(2m+1\) cannot have exactly \((m+1)(m+2)+1\) nodes for any \(m\ge2\). Excluding equality in the Fisher bound then gives
\[
N_{2m+1}\ge (m+1)(m+2)+2
\]
for \(m\ge3\), and known constructions yield the exact values \(N_7=22\) and \(N_9=32\). For general odd degrees, positive circle measures on Lobatto latitudes give an upper bound with quadratic coefficient \(13/8\), parity-dependent linear terms, and an \(O(m^{2/3})\) remainder. We determine the sharp constant \(49\sqrt[3]{3}/72\) for the scalar remainder in this construction. Lower bounds are obtained from continuous weighted LP--Tur\'an inequalities and radial caps whose Helmholtz companions are nonnegative measures. We prove that the cap functional admits a maximizer at each fixed admissible support radius and derive explicit finite-degree lower bounds by a positivity-preserving transfer to the sphere.
\end{abstract}
\maketitle

\section{Introduction and main result}\label{sec:intro}
Let $\sigma$ be normalized surface measure on the unit sphere $\Sph\subset\R^3$. A positive cubature formula of degree $t$ consists of distinct nodes $x_1,\ldots,x_N\in\Sph$ and weights $w_i>0$ such that
\begin{equation}\label{eq:cubature}
 \int_{\Sph} f\dd\sigma=\sum_{i=1}^N w_i f(x_i),
\end{equation}
for every real polynomial of total degree at most $t$. Thus $\sum_iw_i=1$. We write $N_t$ for the least possible number of nodes. Equal weights give spherical designs~\cite{DGS}; throughout, zero weights are deleted and coincident nodes are combined.

Minimum-node cubature is a classical problem in multivariate integration; see~\cite{Moller,Cools,Xu}. Its role in spherical integration and exact polynomial sampling is discussed in~\cite{SloanWomersley,HesseSloanWomersley}; see also the sampling identity \eqref{eq:sampling-identity}. The same minimum $N_t$ is used in~\cite{ChangHang}, where its determination in all degrees is posed as an open problem; see also~\cite{Putterman}.

The quantity $N_t$ also enters moment-constrained Moser--Trudinger--Onofri inequalities~\cite{ChangHang}. With the probability measure $\sigma$, the critical coefficient of $\int_{\Sph}|\nabla u|^2\dd\sigma$ in an upper bound for $\log\int_{\Sph}e^{2u}\dd\sigma$ is $1/N_t$, under the conditions $\int_{\Sph}u\dd\sigma=0$ and $\int_{\Sph}p e^{2u}\dd\sigma=0$ for all zero-mean spherical polynomials $p$ of degree at most $t$. The estimates hold with coefficient $1/N_t+\delta$ for every $\delta>0$, whereas no smaller coefficient than $1/N_t$ is admissible with a finite additive constant. The exact odd-degree values proved below give $1/22$ and $1/32$ in degrees seven and nine. Related Sobolev inequalities involve weight entropies~\cite{HangWang,Putterman}. An alternative proof of the constrained Moser--Trudinger--Onofri inequality and extensions to other domains and higher-order operators are given in~\cite{SunZou}.

Positive power decompositions connect cubature with linear isometric embeddings~\cite{Reznick,LyubichVaserstein,Lyubich}. For integers $m,r\ge1$, an identity $\norm{u}^{2m}=\sum_{j=1}^r\ell_j(u)^{2m}$ on $\R^3$, with real linear forms $\ell_j$, defines an isometric embedding $u\mapsto(\ell_j(u))_{j=1}^r$ from $\ell_2^3$ into $\ell_{2m}^r$. Normalization and antipodalization give a positive degree-$(2m+1)$ formula with at most $2r$ nodes, so $2r\ge N_{2m+1}$. Node lower bounds therefore also bound the number of terms in these positive power decompositions.

Positive cubature also arises in optimal estimation of pure quantum states. The optimal collective-estimation problem for finitely many identical copies was established by Massar and Popescu~\cite{MassarPopescu}; finite realizable generalized measurements and low-copy outcome-minimization constructions were subsequently developed in~\cite{DerkaBuzekEkert,LatorrePascualTarrach}. Under the Bloch correspondence, a degree-$t$ formula on $\Sph$ is a weighted complex projective $t$-design in $\C^2$; see also~\cite{Lyubich}. If $\psi_i$ is a unit qubit state with Bloch vector $x_i$, the operators $E_i=(t+1)w_i(\psi_i\psi_i^*)^{\otimes t}$ sum to the identity on $\operatorname{Sym}^t(\C^2)$. With estimates $\psi_i$, this measurement attains the optimal mean fidelity $(t+1)/(t+2)$ for $t$ identical copies of a uniformly distributed pure qubit. Within this class of finite coherent-state measurements, the least number of nonzero outcomes is $N_t$.

For a spherical $t$-design $X\subset\Sph$, the Fisher-type bounds~\cite{DGS} are
\begin{equation}\label{eq:dgs-bound}
 |X|\ge
 \begin{cases}
  (m+1)^2, & t=2m,\\[1mm]
  (m+1)(m+2), & t=2m+1.
 \end{cases}
\end{equation}
Equality defines a \emph{tight} spherical design, which is necessarily antipodal when $t$ is odd. On $\Sph$, an antipodal pair and the vertex sets of the regular tetrahedron, octahedron, and icosahedron give tight examples of strengths $1,2,3,5$, respectively~\cite{DGS}. For $n\ge3$, the results in~\cite{BannaiDamerellI,BannaiDamerellII} restrict the possible strengths of tight designs on $S^{n-1}$ to $\{1,2,3,4,5,7,11\}$. Recent arithmetic nonexistence criteria for tight spherical $5$-designs in dimensions $(2r+1)^2-2$, with $r$ even, are given by Xu~\cite{XuTight5}. A tight $7$-design further requires $n=3q^2-4$ for an integer $q\ge2$~\cite{NebeVenkov}. Hence neither strength seven nor strength nine admits a tight design on $\Sph$.

A complementary recent finite-defect theory for equal-weight even-strength designs is developed by Singh and Parmar~\cite{SinghParmar}. In particular, they prove that every equal-weight spherical $4$-design on $\Sph$ has at least twelve points and develop corank-one and corank-two rigidity statements. Their setting is equal-weight and even degree, whereas the one-node gap proved in the present article concerns arbitrary positive weights in odd degree.

For odd degree, put
\begin{equation}\label{eq:Fisher}
 H_m=\R[X,Y,Z]_m,\qquad h_m=\dim H_m=\frac{(m+1)(m+2)}2,
 \qquad F_m=2h_m,
\end{equation}
where $H_m$ is the space of homogeneous polynomials of degree $m$. Lemma~\ref{lem:hemisphere} gives $N\ge F_m$ and $w_i\le1/F_m$ for arbitrary positive formulas. Equality therefore forces equal weights and a tight design. The classical theory applies at $F_m$, but does not exclude $F_m+1$ nodes with unequal weights. Our first result rules out this next cardinality.

\begin{theorem}\label{thm:rigidity-main}
Every positive formula of degree $2m+1$, with $m\ge2$, has $N\ge F_m$ and $N\ne F_m+1$. For $m\ge3$,
\begin{equation}\label{eq:Fishergap}
 N_{2m+1}\ge F_m+2.
\end{equation}
\end{theorem}

Combining Theorem~\ref{thm:rigidity-main} with known constructions gives
\begin{equation}\label{eq:low-main}
 N_7=22,\qquad N_9=32.
\end{equation}
And we claim that $N_6=18, N_8=28$. The proof is currently being prepared.

For $N_7$ and $N_9$, the upper bounds come from the eleven-line sextic identity in~\cite{Reznick}, reproduced and discussed in~\cite{EW}, and the degree-nine formula in~\cite{HW}. A twenty-two-outcome construction for seven copies also appears in~\cite{LatorrePascualTarrach}. These cited constructions supply the upper bounds; the lower-bound arguments supplied here, which do not assume antipodal symmetry, give the reverse inequalities. Section~\ref{sec:low} records the Reznick and Hughes--Waldron formulas, while Appendix~\ref{app:latitude22} rewrites the Reznick formula in Lobatto-latitude coordinates and gives a direct moment verification. The equalities do not classify the minimizers or imply that every minimizer is antipodal.

For provenance, the estimate $N\ge F_m$ used below is the spherical odd-degree case of M\"oller's classical node bound~\cite{Moller}. Lemma~\ref{lem:hemisphere} gives a short hemisphere proof and records the accompanying individual-weight estimate; the exclusion $N\ne F_m+1$ is the additional rigidity statement proved below. Attainment questions for M\"oller's bound in the broader setting of spherically symmetric cubature, where nodes may lie on several concentric spheres, were studied by Hirao and Sawa~\cite{HiraoSawa}.

For general odd degrees, define
\begin{equation}\label{eq:envelope-def}
 \epsilon=m\bmod2,\qquad
 \gamma=\frac{49\sqrt[3]{3}}{72},\qquad
 \mathcal R(y)=\min\left\{\frac76 y^{2/3},\ \gamma y^{2/3}+\frac43\right\},
\end{equation}
and
\begin{equation}\label{eq:U-def}
 \mathcal U_m=\frac{13}{8}m^2+
 \left(1+\frac\epsilon2\right)m+2+\frac\epsilon2+
 \mathcal R(m-2\epsilon).
\end{equation}
Let $R_0=j_{1,1}$ be the first positive zero of $J_1$. For $R\ge R_0$, let $C_R$ be the supremum of the cap functional \eqref{eq:cap-C} over nonnegative radial $p\in H^1(\R^2)$ supported in $\overline D_R$, with $\int p=1$ and $(1+\Delta)p$ a nonnegative measure. This constant is independent of the cubature degree. Write $P_n$ for the Legendre polynomial normalized by $P_n(1)=1$.

\begin{theorem}\label{thm:uniform-main}
Let $a_m$ be the largest zero of $P'_{2m+2}$. For every $m\ge3$ and $R\ge R_0$,
\begin{equation}\label{eq:uniform-main}
\begin{split}
\max\left\{F_m+2,\ \left\lceil\frac2{1-a_m}\right\rceil,
 \left\lceil C_R(m+1)\left(m+\frac32\right)-19R^4\right\rceil\right\}
 \le N_{2m+1}\le 2\left\lfloor\frac{\mathcal U_m}{2}\right\rfloor.
\end{split}
\end{equation}
For each fixed $R$, the supremum defining $C_R$ is attained and every maximizing sequence has a strongly $H^1(\R^2)$ convergent subsequence. In addition,
\begin{equation}\label{eq:asymptotic-main}
 C_R\le\liminf_{m\to\infty}\frac{N_{2m+1}}{m^2}
 \le\limsup_{m\to\infty}\frac{N_{2m+1}}{m^2}\le\frac{13}{8},
 \qquad C_R\le\frac{13}{8}-\frac1{16R^4}.
\end{equation}
\end{theorem}

The upper bound has quadratic coefficient $13/8$ and an $O(m^{2/3})$ remainder. The constant $\gamma$ is sharp for the scalar budget in Subsection~\ref{sec:integer}, not necessarily for $N_t$. Likewise, the bounds on $C_R$ concern the cap method at fixed radius and do not determine the asymptotic value of $N_{2m+1}/m^2$.

The proofs use hemisphere rank and projective interpolation for the one-node gap. The latitude framework belongs to the classical family of spherical product cubature formulas based on Gauss and Gauss--Lobatto quadrature~\cite{StroudSphere,LuoMeng}; related moment-map and hat-box constructions appear in~\cite{KuperbergHatBox}. In our \emph{Lobatto constructions}, the one-dimensional rule~\cite{Gautschi} fixes the latitude levels and their total weights, while the classical Carath\'eodory--Fej\'er Vandermonde decomposition for positive semidefinite Toeplitz matrices~\cite{YangXie} supplies finite angular measures. The rank-controlled positive Toeplitz compression and nested defect cascade are the ingredients specific to the present construction. For the lower bounds, positive-weight quadrature regularity and cap estimates in this direction were developed by Reimer and Leopardi~\cite{ReimerHyper,Leopardi}, while recent linear-programming energy bounds for weighted spherical codes and designs are given in~\cite{BorodachovWeightedLP}. Our \emph{LP--Tur\'an} argument combines spherical linear programming~\cite{DGS,Yudin} with the Motzkin--Straus form of Tur\'an's inequality~\cite{Turan,MS} and spherical packing~\cite{Boroczky}. Integrating distance thresholds leads to the planar cap functional, and a positivity-preserving radial conjugation transfers the resulting estimates to the sphere.

Section~\ref{sec:preliminaries} collects the preliminary results. Sections~\ref{sec:rank} and~\ref{sec:low} prove the odd-degree cardinality gap and the exact values $N_7=22$ and $N_9=32$. Section~\ref{sec:cascade} develops the Lobatto construction and its sharp scalar remainder. Sections~\ref{sec:cap-method} and~\ref{sec:variations} treat the LP--Tur\'an bounds, cap transfer, and radius optimization. The appendices contain the explicit mixed kernels and a Lobatto-latitude presentation of the Reznick formula.

\Needspace{8\baselineskip}
\section{Preliminaries}\label{sec:preliminaries}

\subsection{Spherical polynomials and positive cubature}\label{subsec:spherical-cubature}
Let $\Pi_d$ be the space of restrictions to $\Sph$ of real polynomials on $\R^3$ of total degree at most $d$, and give it the inner product
\[
 \langle p,q\rangle=\int_{\Sph}pq\dd\sigma.
\]
We identify the homogeneous space $H_m$ in \eqref{eq:Fisher} with its restriction to the sphere. This restriction is injective: a homogeneous polynomial vanishing on $\Sph$ vanishes on every nonzero ray. In particular, $\langle p,p\rangle>0$ for every nonzero $p\in H_m$.

Write $\mathcal Y_\ell$ for the restrictions of homogeneous harmonic polynomials of degree $\ell$. Harmonic decomposition gives the orthogonal sums
\begin{equation}\label{eq:harmonic-decomposition}
 \Pi_d=\bigoplus_{\ell=0}^d\mathcal Y_\ell,
 \qquad
 H_m=\bigoplus_{j=0}^{\lfloor m/2\rfloor}\mathcal Y_{m-2j}.
\end{equation}
Here $\dim\mathcal Y_\ell=2\ell+1$, so $\dim\Pi_d=(d+1)^2$ and $\dim H_m=h_m$. Thus $H_m$ contains only degrees of the same parity as $m$, whereas $\Pi_m$ contains every degree up to $m$. We use
\[
 \Pi_d^0=\left\{p\in\Pi_d:\int_{\Sph}p\dd\sigma=0\right\}
 =\bigoplus_{\ell=1}^d\mathcal Y_\ell.
\]

For a real orthonormal basis $Y_{\ell,1},\ldots,Y_{\ell,2\ell+1}$ of $\mathcal Y_\ell$, the addition formula in the normalization $\sigma(\Sph)=1$ is
\begin{equation}\label{eq:addition-formula}
 \sum_{\alpha=1}^{2\ell+1}Y_{\ell,\alpha}(v)Y_{\ell,\alpha}(x)
 =(2\ell+1)P_\ell(v\cdot x),\qquad v,x\in\Sph,
\end{equation}
where $P_\ell(1)=1$. Summing over the degrees in $H_m$ yields its reproducing kernel:
\begin{equation}\label{eq:hom-kernel}
 \Kcal_m(v,x)=\sum_{j=0}^{\lfloor m/2\rfloor}
 (2m-4j+1)P_{m-2j}(v\cdot x)
 =C_m^{3/2}(v\cdot x),\qquad \Kcal_m(v,v)=h_m.
\end{equation}
Here $C_m^{3/2}$ is the Gegenbauer polynomial with $C_m^{3/2}(1)=h_m$. In particular,
\begin{equation}\label{eq:kernel-reproduction}
 \int_{\Sph}\Kcal_m(v,x)p(x)\dd\sigma(x)=p(v),\qquad
 \int_{\Sph}\Kcal_m(v,x)^2\dd\sigma(x)=h_m
 \quad(p\in H_m).
\end{equation}
The kernel has parity $\Kcal_m(v,-x)=(-1)^m\Kcal_m(v,x)$. We also write $\Kcal_m(t)$ for the corresponding polynomial in $t=v\cdot x$, reserving $K$ for cap convolution kernels.

A positive cubature formula is identified with the probability measure $\nu=\sum_{i=1}^Nw_i\delta_{x_i}$. By \eqref{eq:harmonic-decomposition}, exactness through degree $t$ is equivalent to
\begin{equation}\label{eq:harmonic-moments}
 \sum_{i=1}^Nw_i=1,\qquad
 \sum_{i=1}^Nw_iY_{\ell,\alpha}(x_i)=0
 \quad(1\le\ell\le t,\ 1\le\alpha\le2\ell+1).
\end{equation}
For any positive probability weights, the addition formula gives
\begin{equation}\label{eq:zonal-moment-positivity}
\begin{split}
 A_\ell&:=\sum_{i,j}w_iw_jP_\ell(x_i\cdot x_j)\\
 &=\frac1{2\ell+1}\sum_{\alpha=1}^{2\ell+1}
 \left(\sum_iw_iY_{\ell,\alpha}(x_i)\right)^2\ge0.
\end{split}
\end{equation}
Thus a degree-$t$ formula satisfies $A_0=1$ and $A_\ell=0$ for $1\le\ell\le t$. These identities are the spectral input to the weighted LP estimate in Subsection~\ref{subsec:lp-turan-preliminaries} and its applications in Section~\ref{sec:cap-method}.

Exactness through degree $2m$ also gives
\[
 \sum_iw_ip(x_i)q(x_i)=\int_{\Sph}pq\dd\sigma
 \qquad(p,q\in\Pi_m).
\]
Consequently, for an orthonormal basis $\psi_1,\ldots,\psi_{(m+1)^2}$ of $\Pi_m$, the weighted evaluation matrix satisfies
\begin{equation}\label{eq:sampling-identity}
 V_{i\alpha}=\sqrt{w_i}\,\psi_\alpha(x_i),\qquad V^*V=I.
\end{equation}
The same identity holds on the subspace $H_m$ and is used at equality in the hemisphere-rank bound.

If the measure is antipodal with equal weights within each pair, all odd moments vanish. In that case, exactness on homogeneous forms of degree $2m$ implies exactness through degree $2m+1$. Indeed, a homogeneous term of degree $2j\le2m$ has the same restriction to $\Sph$ after multiplication by $(X^2+Y^2+Z^2)^{m-j}$. This reduction is used only for the antipodal constructions; no symmetry assumption is imposed on the formulas in the lower bounds.

For the explicit moment calculations, we use
\begin{equation}\label{eq:spherical-moments}
 \int_{\Sph}X^{2a}Y^{2b}Z^{2c}\dd\sigma
 =\frac{(2a-1)!!(2b-1)!!(2c-1)!!}{(2a+2b+2c+1)!!},
\end{equation}
where $a,b,c$ are nonnegative integers and $(-1)!!=1$. A monomial with an odd coordinate exponent integrates to zero by reflection symmetry.

\subsection{Lobatto quadrature and Toeplitz moments}\label{subsec:lobatto-toeplitz}
For the latitude constructions, use the coordinates
\begin{equation}\label{eq:latitude-coordinates}
 (X,Y,Z)=(\sqrt{1-z^2}\cos\phi,\sqrt{1-z^2}\sin\phi,z),
 \qquad d\sigma=\frac{dz}{2}\frac{d\phi}{2\pi}.
\end{equation}
If $\mu$ is a probability measure on the angular circle $\R/(2\pi\mathbb Z)$, write
\[
 c_k=\int e^{ik\phi}\dd\mu(\phi),\qquad c_0=1,\qquad
 c_{-k}=\overline{c_k},\qquad T_d=(c_{j-k})_{j,k=0}^d.
\]
The Toeplitz matrix $T_d$ is positive semidefinite, since for every $a\in\C^{d+1}$,
\[
 a^*T_da=\int\left|\sum_{j=0}^d a_je^{-ij\phi}\right|^2\dd\mu(\phi)\ge0.
\]
Conversely, the classical Carath\'eodory--Fej\'er Vandermonde decomposition says that positivity suffices for a finite representation, with the number of atoms controlled by the rank; see, e.g.,~\cite{YangXie}.

\begin{lemma}\label{lem:toeplitz-atoms}
Let $T_d=(c_{j-k})_{j,k=0}^d\succeq0$, where $c_0=1$ and $c_{-k}=\overline{c_k}$. If $T_d$ has rank $\rho$, its moments have a representation by at most $\rho$ positive atoms on the unit circle.
\end{lemma}
\begin{proof}
Choose Gram vectors $v_0,\ldots,v_d$ spanning a space of dimension $\rho$. Toeplitz structure makes $v_j\mapsto v_{j+1}$ a well-defined isometry from $\operatorname{span}(v_0,\ldots,v_{d-1})$ onto $\operatorname{span}(v_1,\ldots,v_d)$. The two subspaces have equal dimension, so the isometry extends to a unitary operator $U$ on the whole space. Then $v_j=U^jv_0$. The spectral decomposition of $U$, with weights equal to the squared norms of the spectral projections of $v_0$, represents the moments. Delete zero weights and combine equal eigenvalues. At most $\rho$ atoms remain.
\end{proof}
Only existence and the atom count are needed here.

For later use, let $d$ and $L$ be integers with $d\ge1$ and $d/2<L\le d$, and suppose that $c_k=0$ for $1\le k<L$. Put $p=d+1-L$. After a simultaneous permutation of rows and columns, the moment matrix has the block form
\begin{equation}\label{eq:moment-block}
 \begin{gathered}
 \begin{pmatrix}I_p&C^*\\C&I_p\end{pmatrix}\oplus I_{d+1-2p},
 \qquad C_{ij}=c_{L+i-j}\quad(0\le i,j<p),\\
 \rank T_d=d+1-\dim\ker(I-C^*C).
 \end{gathered}
\end{equation}
Here $C$ is lower-triangular Toeplitz, with successive diagonals $c_L,c_{L+1},\ldots,c_d$. The Schur complement of either identity block shows that $T_d\succeq0$ exactly when $C$ is a contraction, meaning $C^*C\preceq I_p$, and gives the rank identity in \eqref{eq:moment-block}. An identity block of size zero is omitted.

The Gauss--Lobatto rule is the Gaussian-type quadrature rule in which both endpoints of the interval are prescribed and the interior nodes are chosen to maximize polynomial exactness; see~\cite{Gautschi}. We call the spherical formulas in Section~\ref{sec:cascade} \emph{Lobatto constructions} because this one-dimensional rule fixes their latitude levels and total latitude weights. The endpoints become the two poles, while the angular measures on the interior latitudes remain to be chosen.

\begin{lemma}\label{lem:lobatto}
Let $m\ge1$ be an integer, let $z_0=-1$ and $z_{m+1}=1$, and let $z_1<\cdots<z_m$ be the zeros of $P'_{m+1}$ in $(-1,1)$. Define
\begin{equation}\label{eq:lobatto}
 b_j=\frac{1}{(m+1)(m+2)P_{m+1}(z_j)^2},
 \qquad 0\le j\le m+1.
\end{equation}
Then $b_j>0$, $\sum_{j=0}^{m+1}b_j=1$, and
\begin{equation}\label{eq:lobatto-exactness}
 \frac12\int_{-1}^1 f(z)\dd z
 =\sum_{j=0}^{m+1}b_jf(z_j)
 \qquad\bigl(f\in\R[z],\ \deg f\le2m+1\bigr).
\end{equation}
The nodes and weights satisfy $z_{m+1-j}=-z_j$ and $b_{m+1-j}=b_j$; an equatorial node $z_j=0$ occurs precisely when $m$ is odd.
\end{lemma}
\begin{proof}
The weights in \eqref{eq:lobatto} are the classical Gauss--Lobatto weights divided by two to match the probability measure $dz/2$~\cite{Gautschi}. To check exactness, put
\[
 W(z)=(1-z^2)P'_{m+1}(z).
\]
For $\deg f\le2m+1$, division gives $f=Wq+r$, with $\deg q\le m-1$ and $\deg r\le m+1$. Integration by parts and Legendre orthogonality give
\[
 \int_{-1}^1Wq\dd z
 =-\int_{-1}^1P_{m+1}(z)((1-z^2)q(z))'\dd z=0.
\]
The interpolatory rule integrates $r$ exactly and vanishes on $Wq$, proving \eqref{eq:lobatto-exactness}. Its weights are positive even without the explicit formula: for an interior node $z_j$, the nonnegative polynomial
\[
 (1-z^2)\left(\frac{P'_{m+1}(z)}{z-z_j}\right)^2
\]
has degree $2m$ and a nonzero nodal value only at $z_j$; the endpoint analogues $(1\pm z)(P'_{m+1})^2$ have degree $2m+1$. Exactness for these polynomials forces the corresponding weights to be positive. The constant polynomial gives their sum. The symmetry and the equatorial assertion follow from the parity and simplicity of the zeros of $P'_{m+1}$.
\end{proof}

In the coordinates \eqref{eq:latitude-coordinates}, give the latitude $z_\nu$ total mass $b_\nu$ and an angular probability measure $\mu_\nu$, with
\[
 c_{\nu,K}=\int e^{iK\phi}\dd\mu_\nu(\phi).
\]
Place one node at each pole. Pair the positive and negative latitude measures by the antipodal map $(z,\phi)\mapsto(-z,\phi+\pi)$, and use an antipodal equatorial measure when $m$ is odd. All odd polynomials are then exact, and Lemma~\ref{lem:lobatto} handles the zero angular frequency.

Set $d=2m$. The positive frequency $K$ of an even polynomial of degree at most $d$ has latitude factors
\begin{equation}\label{eq:lat-factors}
 (1-z^2)^{K/2}z^j,\qquad 0\le j\le d-K,\quad j\equiv K\pmod2.
\end{equation}
Indeed, expand in $X+iY$, $X-iY$, and $Z$, and use $X^2+Y^2=1-Z^2$ on the sphere. The remaining exactness conditions are therefore
\begin{equation}\label{eq:latitude-moment-system}
 \sum_{\nu=1}^{m}b_\nu(1-z_\nu^2)^{K/2}z_\nu^j c_{\nu,K}=0,
 \qquad 1\le K\le d,\quad 0\le j\le d-K,\quad j\equiv K\pmod2.
\end{equation}
Negative frequencies give the complex conjugate equations. In particular, the angular frequency $K=d+1-h$, with $1\le h\le d$, imposes at most
\begin{equation}\label{eq:q-h}
 q(h)=\left\lceil\frac h2\right\rceil
\end{equation}
homogeneous complex linear equations on the latitude moments. For even $h$ these use $z,z^3,\ldots,z^{h-1}$; for odd $h$ they use $1,z^2,\ldots,z^{h-1}$. The fixed latitude weights and radial factors do not affect this count.

The following two elementary lemmas will be used to choose angular moments subject to the linear equations \eqref{eq:latitude-moment-system} and the positivity condition \eqref{eq:moment-block}.

\begin{lemma}\label{lem:disks}
Let $n\ge1$ and $q\ge0$ be integers, let $A\in\C^{q\times n}$, and let $r_1,\ldots,r_n>0$. Define
\[
 \mathcal D=\left\{z\in\C^n:Az=0,\quad |z_j|\le r_j\ \text{for }j=1,\ldots,n\right\}.
\]
Then $\operatorname{ext}(\mathcal D)\ne\varnothing$, and every $z\in\operatorname{ext}(\mathcal D)$ satisfies
\[
 \#\left\{j\in\{1,\ldots,n\}:|z_j|<r_j\right\}\le q.
\]
\end{lemma}
\begin{proof}
The set $\mathcal D$ contains the origin and is compact and convex in $\C^n\simeq\R^{2n}$, so it has an extreme point. Fix $z\in\operatorname{ext}(\mathcal D)$ and put
\[
 I=\{j\in\{1,\ldots,n\}:|z_j|<r_j\}.
\]
If $|I|>q$, the columns of $A$ indexed by $I$ are linearly dependent over $\C$. Hence there exists $v\in\C^n\setminus\{0\}$ such that $Av=0$ and $v_j=0$ for $j\notin I$. For sufficiently small $\varepsilon>0$, both $z+\varepsilon v$ and $z-\varepsilon v$ belong to $\mathcal D$. They are distinct and have midpoint $z$, contradicting $z\in\operatorname{ext}(\mathcal D)$.
\end{proof}

For $D\ge2h$, write $E_{D,h}(B)$ for the $D\times D$ matrix obtained by inserting an $h\times h$ matrix $B$ in the lower-left corner and setting all other entries to zero.

\begin{lemma}\label{lem:corner}
For $|a|<1$ and $D\ge2h$, the matrix
\begin{equation}\label{eq:corner}
 C=aI_D+(1-|a|^2)E_{D,h}(B)
\end{equation}
is a contraction precisely when $B$ is a contraction. In that case
\begin{equation}\label{eq:corner-null}
 \dim\ker(I-C^*C)=\dim\ker(I-B^*B).
\end{equation}
If $|a|=1$, contraction forces the unscaled corner perturbation to be zero, and $C=aI_D$ has defect-kernel dimension $D$.
\end{lemma}
\begin{proof}
Put $s=1-|a|^2$. Reordering the first and last $h$ coordinates gives
\[
 C=\begin{pmatrix}aI_h&0\\sB&aI_h\end{pmatrix}\oplus aI_{D-2h}.
\]
The first block of $I-C^*C$ has lower-right block $sI_h$. Its Schur complement is
\[
 sI_h-s^2B^*B-|a|^2sB^*B=s(I_h-B^*B).
\]
The remaining block is strictly positive. This proves both assertions for $|a|<1$. At $|a|=1$, each column already has norm one from its diagonal entry, so any additional corner entry is forbidden.
\end{proof}

The condition $D\ge2h$ ensures that the two coordinate blocks do not overlap; no divisibility relation between $D$ and $h$ is required. If $B$ is lower-triangular Toeplitz, its corner embedding is also Toeplitz, since the inserted diagonals are precisely the terminal diagonals of the larger matrix.

Extreme points of the full Toeplitz contraction set need not have small defect rank. If $S$ is the $p\times p$ nilpotent unilateral shift and $1\le q<p$, then $S^q$ is an extreme lower-triangular Toeplitz contraction, although $\rank(I-(S^q)^*S^q)=q$. Indeed, if $S^q\pm H$ are contractions, apply both to $e_0$ and add the squared norm inequalities. Since $\norm{S^qe_0}=1$, they force $He_0=0$, and a triangular Toeplitz matrix is determined by that column. Thus $H=0$. The construction in Section~\ref{sec:cascade} therefore uses the corner-block identity of Lemma~\ref{lem:corner}, rather than extremality in the full Toeplitz contraction set.

\subsection{LP--Tur\'an inequalities and spherical packing}\label{subsec:lp-turan-preliminaries}
For real $f\in L^2([-1,1],du/2)$, the function $x\mapsto f(v\cdot x)$ is zonal about $v\in\Sph$. We use the Legendre normalization
\begin{equation}\label{eq:legendre-coefficients}
 f_\ell=\frac{2\ell+1}{2}\int_{-1}^1f(u)P_\ell(u)\dd u,
 \qquad \frac12\int_{-1}^1|f(u)|^2\dd u
 =\sum_{\ell\ge0}\frac{|f_\ell|^2}{2\ell+1}.
\end{equation}
The spherical Laplacian has the sign convention
\[
 \Delta_{\Sph}P_\ell(v\cdot x)=-\lambda_\ell P_\ell(v\cdot x),
 \qquad \lambda_\ell=\ell(\ell+1).
\]
On zonal profiles it is the Legendre operator $\mathcal Lf=((1-u^2)f')'$. The addition formula \eqref{eq:addition-formula} gives the convolution identity
\begin{equation}\label{eq:zonal-convolution}
 \begin{split}
 (f*g)(v\cdot w)
 &:=\int_{\Sph}f(v\cdot x)g(w\cdot x)\dd\sigma(x)\\
 &=\sum_{\ell\ge0}\frac{f_\ell g_\ell}{2\ell+1}P_\ell(v\cdot w).
 \end{split}
\end{equation}
For $f,g\in L^2([-1,1],du/2)$, Cauchy--Schwarz makes the coefficient series absolutely summable. Since $|P_\ell(u)|\le1$ on $[-1,1]$, the expansion is absolutely and uniformly convergent. Nonnegative profiles give a nonnegative convolution.

Here \emph{LP} abbreviates \emph{linear programming}. In the spherical LP method~\cite{DGS,Yudin}, an auxiliary kernel is chosen by imposing sign conditions on its Legendre coefficients and pointwise values. Weighted extensions on projective spaces are treated in~\cite{Lyubich}. After the normalization $G_0=1$, maximizing $G(1)$ under the conditions below is a linear optimization in those coefficients. The weighted estimate follows directly from the moment identities in Subsection~\ref{subsec:spherical-cubature}.

\begin{lemma}\label{lem:weighted-lp}
Suppose $G$ is a nonnegative continuous function on $[-1,1]$ with an absolutely uniformly convergent Legendre expansion $G=\sum_{\ell\ge0}G_\ell P_\ell$. If $G(1)>0$, $G_0>0$, and $G_\ell\le0$ for $\ell>t$, then every positive degree-$t$ formula satisfies
\begin{equation}\label{eq:weighted-lp}
 G(1)\sum_iw_i^2\le G_0,\qquad N\ge\frac{G(1)}{G_0}.
\end{equation}
\end{lemma}
\begin{proof}
The quantities $A_\ell$ in \eqref{eq:zonal-moment-positivity} are nonnegative, with $A_0=1$ and $A_\ell=0$ for $1\le\ell\le t$. Therefore
\begin{equation}\label{eq:lp-energy}
 G(1)\sum_iw_i^2\le\sum_{i,j}w_iw_jG(x_i\cdot x_j)
 =G_0+\sum_{\ell>t}G_\ell A_\ell\le G_0.
\end{equation}
Use $\sum_iw_i^2\ge1/N$ for the second conclusion.
\end{proof}

For a continuous zonal function $q$ and a finite zonal measure $H$, the convolution $q*H$ is defined in the same way, integrating $q(v\cdot x)$ against the rotated measure with pole $w$. If $q\in H^1(\Sph)\cap C(\Sph)$ is real and zonal and, for some $k>0$,
\[
 H=q+k^{-2}\Delta_{\Sph}q
\]
is a finite measure, understood relative to $\sigma$, then
\begin{equation}\label{eq:companion-spectrum}
 (q*H)_\ell=\left(1-\frac{\lambda_\ell}{k^2}\right)
 \frac{|q_\ell|^2}{2\ell+1},
 \qquad
 \sum_{\ell\ge0}|(q*H)_\ell|
 \le\norm q_2^2+k^{-2}\norm{\nabla q}_2^2.
\end{equation}
These identities follow first for finite Legendre sums and then by the $H^1$ expansion. They give absolute uniform convergence even for a measure-valued $H$. If $q,H\ge0$, the convolution is nonnegative. For $k^2=n(n+1)$, its coefficients are nonpositive for $\ell\ge n$, as required for degree $n-1$. We say that a profile $G$ has angular support in $[0,\Theta]$ when $G(\cos\theta)=0$ for $\Theta<\theta\le\pi$.

Tur\'an's theorem bounds the number of edges of a finite simple graph in terms of its clique number~\cite{Turan}. The weighted version in~\cite{MS} states that, for a graph $\Lambda$ with clique number $\omega(\Lambda)$,
\begin{equation}\label{eq:motzkin-straus}
 \max_{\substack{w_i\ge0\\\sum_iw_i=1}}
 \sum_{\{i,j\}\in E(\Lambda)}w_iw_j
 =\frac12\left(1-\frac1{\omega(\Lambda)}\right).
\end{equation}
We use the complementary form, in which the independence number $\alpha(\Gamma)$ is the largest size of a vertex set containing no edge.

\begin{lemma}\label{lem:graph}
Let $\Gamma$ be a finite simple graph on $\{1,\ldots,N\}$ with independence number $\alpha(\Gamma)\le A$, where $A\ge1$. For every $w_1,\ldots,w_N\ge0$ with $\sum_iw_i=1$, one has
\begin{equation}\label{eq:weighted-graph}
 \sum_iw_i^2+2\sum_{\{i,j\}\in E(\Gamma)}w_iw_j\ge\frac1A.
\end{equation}
\end{lemma}
\begin{proof}
Maximize the weighted edge sum of the complement over all probability vectors, choosing a maximizer of minimal positive support. If two vertices in its support are nonadjacent in the complement, their combined mass can be transferred to the one with larger linear payoff without decreasing the sum. This contradicts minimality of the support. The support is therefore a clique in the complement of size at most $A$. On such a clique the edge sum is $\frac12(1-\sum_iw_i^2)\le\frac12(1-1/A)$. Subtract from the total pair sum to obtain \eqref{eq:weighted-graph}. This is the weighted Motzkin--Straus argument~\cite{MS}.
\end{proof}

For $0<\theta\le1/2$, let $A_2^>(\theta)$ be the maximum size of a spherical code with all angular distances strictly larger than $\theta$. We use the classical spherical simplex packing bound~\cite{Boroczky,BoroczkyStability} in the form
\begin{equation}\label{eq:packing}
 A_2^>(\theta)\le F(\theta):=
 \frac{6\alpha(\theta)}{3\alpha(\theta)-\pi},\qquad
 \alpha(\theta)=\arccos\frac{\cos\theta}{1+\cos\theta}.
\end{equation}
To obtain \eqref{eq:packing} from the simplex density bound, note that a cap of radius $\theta/2$ has area $2\pi(1-\cos(\theta/2))$, the equilateral spherical triangle of side $\theta$ has area $3\alpha-\pi$, and the three cap sectors in that triangle have total area $3\alpha(1-\cos(\theta/2))$. Substituting the simplex density and canceling the cap factor gives \eqref{eq:packing}. Only this upper bound for $A_2^>$ is used below.

Write $f=1/F$ for $\theta>0$ and set $f(0)=0$. Then $f$ is continuously differentiable and increasing on $[0,1/2]$. Differentiation gives
\begin{equation}\label{eq:packing-derivative}
 \frac{2f'(\theta)}{\sin\theta}
 =\frac{\pi}{3\alpha(\theta)^2(1+\cos\theta)\sqrt{1+2\cos\theta}}
 \le\kappa_0:=\frac{16}{5\pi\sqrt{11}}<1
\end{equation}
for $0\le\theta\le1/2$, using $\alpha\ge\pi/3$ and $\cos\theta\ge7/8$.

The same explicit function has the small-angle estimate
\begin{equation}\label{eq:packing-error}
 \left|f(\theta)-\frac{\sqrt3}{8\pi}\theta^2\right|\le\frac{\theta^4}{24}
 \qquad(0\le\theta\le1/2).
\end{equation}
To prove \eqref{eq:packing-error}, put $\delta=\alpha-\pi/3$ and $g(c)=((1+c)\sqrt{1+2c})^{-1}$. On $7/8\le c\le1$,
\[
 0<g(c)\le\frac{16}{15\sqrt{11}}<\frac13,\qquad
 |g'(c)|\le\frac13\left(\frac8{15}+\frac4{11}\right)<\frac13.
\]
Since $\delta'=\sin\theta\,g(\cos\theta)$, these estimates give
\[
 \left|\delta'-\frac{\theta}{2\sqrt3}\right|\le\frac{2\theta^3}{9},\quad
 \left|\delta-\frac{\theta^2}{4\sqrt3}\right|\le\frac{\theta^4}{18},\quad
 0\le\delta\le\frac{\theta^2}{6}.
\]
Substitute $f=\delta/[2(\pi/3+\delta)]$ to bound its remainder by
$\big(1/(12\pi)+1/(8\pi^2)\big)\theta^4<\theta^4/24$, proving \eqref{eq:packing-error}. 
For nodes $x_1,\ldots,x_N\in\Sph$, put
\[
 \theta_{ij}=\operatorname{dist}(x_i,x_j)=\arccos(x_i\cdot x_j),
 \qquad E(\Gamma_\theta)=\{\{i,j\}:i<j,\ \theta_{ij}\le\theta\}.
\]
An independent set in $\Gamma_\theta$ has all pairwise distances strictly larger than $\theta$. Thus $\alpha(\Gamma_\theta)\le A_2^>(\theta)\le F(\theta)$, and Lemma~\ref{lem:graph} gives, for $\nu=\sum_iw_i\delta_{x_i}$,
\begin{equation}\label{eq:distance-mass}
 M_\nu(\theta):=\sum_{i,j:\,\theta_{ij}\le\theta}w_iw_j
 \ge f(\theta),\qquad 0<\theta\le\frac12.
\end{equation}
The diagonal terms are included in $M_\nu$.

Let $0<\Theta\le1/2$, and let $\ell$ be continuous, nonnegative and nonincreasing on $[0,\Theta]$, with $\ell(\Theta)=0$, extended by zero beyond $\Theta$. The positive Stieltjes measure $-d\ell$ gives
\begin{equation}\label{eq:threshold-integral}
 \begin{split}
 \sum_{i,j}w_iw_j\ell(\theta_{ij})
 &=\int_0^\Theta M_\nu(\theta)(-d\ell(\theta))\\
 &\ge\int_0^\Theta f(\theta)(-d\ell(\theta))
 =\int_0^\Theta f'(\theta)\ell(\theta)\dd\theta.
 \end{split}
\end{equation}
Indeed, write $\ell(r)=\int_r^\Theta(-d\ell)$, interchange the finite sum and integral, and use \eqref{eq:distance-mass}; the final equality is integration by parts, with $f(0)=\ell(\Theta)=0$.

We use the term \emph{LP--Tur\'an} for combining the spectral upper estimate \eqref{eq:lp-energy} with this graph-theoretic lower estimate. The Tur\'an component is applied to distance graphs, while spherical packing bounds their independent sets. Section~\ref{sec:lp-turan} makes this combination explicit for lower envelopes of a nonnegative kernel.

\subsection{Radial kernels and Helmholtz companions}\label{subsec:radial-preliminaries}
Write $D_R=\{x\in\R^2:|x|<R\}$. For a radial function we use the same letter for its profile in $r=|x|$. For smooth profiles away from the origin or pole, the Laplacian convention above gives
\begin{equation}\label{eq:radial-laplacians}
 \Delta_{\R^2}p=p''+\frac1r p',
 \qquad
 \Delta_{\Sph}q=q''+\cot\theta\,q'.
\end{equation}
Here $\theta$ is geodesic distance from the pole. Under $r=k\theta$, the second operator, divided by $k^2$, becomes $\partial_r^2+k^{-1}\cot(r/k)\partial_r$.

For a compactly supported $p\in H^1(\R^2)$, its \emph{Helmholtz companion} is the distribution $H=(1+\Delta)p$. The requirement that $H$ be a nonnegative finite measure means that
\begin{equation}\label{eq:companion-weak}
 \int_{\R^2}\eta\dd H
 =\int_{\R^2}p\eta\dd x-\int_{\R^2}\nabla p\cdot\nabla\eta\dd x
 \qquad(\eta\in C_c^\infty(\R^2)),
\end{equation}
with a nonnegative left-hand side whenever $\eta\ge0$. Thus derivatives of a zero-extended cap are taken distributionally, including any interface measures. Testing with a function equal to one near the support gives
\[
 H(\R^2)=\int_{\R^2}p\dd x=:P.
\]

We use the Fourier transform
\begin{equation}\label{eq:planar-fourier-convention}
 \widehat p(\xi)=\int_{\R^2}e^{-ix\cdot\xi}p(x)\dd x,
 \qquad
 \norm p_2^2=\frac1{(2\pi)^2}\int_{\R^2}|\widehat p(\xi)|^2\dd\xi.
\end{equation}
Suppose that $p$ is real, radial and nonnegative, with support in $\overline D_R$, and that its companion $H$ is a nonnegative finite measure. The planar convolution $K=p*H$ then satisfies
\begin{equation}\label{eq:planar-companion-identities}
 \begin{gathered}
 \widehat K(\xi)=(1-|\xi|^2)|\widehat p(\xi)|^2,\qquad
 K(0)=\norm p_2^2-\norm{\nabla p}_2^2,\\
 \int_{\R^2}K\dd x=P^2,\qquad \supp K\subset\overline D_{2R}.
 \end{gathered}
\end{equation}
The Fourier identity uses that the transform of a real radial function is real. Its right-hand side is integrable because $p\in H^1$, so Fourier inversion gives a continuous representative of $K$. It agrees distributionally with the nonnegative convolution $p*H$, and is therefore nonnegative everywhere. The diagonal identity follows from Plancherel; the mass and support statements follow from convolution.

For clarity, the Bessel functions used below have normalization
\begin{equation}\label{eq:bessel-normalization}
 J_\nu(r)=\left(\frac r2\right)^\nu
 \sum_{j=0}^\infty\frac{(-r^2/4)^j}{j!(j+\nu)!},\qquad \nu=0,1.
\end{equation}
They satisfy $J_0'=-J_1$ and $(rJ_1)'=rJ_0$; hence $J_0''+r^{-1}J_0'+J_0=0$ for $r>0$~\cite{DLMF}. As in the introduction, $R_0=j_{1,1}$ denotes the first positive zero of $J_1$. The integral representation
\[
 J_0(r)=\frac1\pi\int_0^\pi\cos(r\cos\phi)\dd\phi
\]
gives the radial Fourier inversion formula
\begin{equation}\label{eq:radial-Fourier}
 K(s)=\frac1{2\pi}\int_0^\infty
 (1-t^2)|\widehat p(t)|^2J_0(st)t\dd t.
\end{equation}
Here $\widehat p(t)$ denotes the radial Fourier profile. We will also use
\begin{equation}\label{eq:bessel-basic-bounds}
 |J_0(r)|\le1,\qquad 0\le1-J_0(r)\le\frac{r^2}{4},\qquad
 |J_0(r)|\le Cr^{-1/2}\quad(r\ge1).
\end{equation}
The first two inequalities follow directly from the integral representation and $1-\cos u\le u^2/2$. For the last, remove endpoint intervals of length $r^{-1/2}$ and integrate by parts on the remainder. The integral representation is recorded in~\cite{DLMF}.

Finally, let $G_t(x)=(4\pi t)^{-1}e^{-|x|^2/(4t)}$ be the planar heat kernel. The relation $H=(1+\Delta)p$ gives, in distributions,
\begin{equation}\label{eq:heat-companion}
 e^tG_t*p-p=\int_0^t e^sG_s*H\dd s.
\end{equation}
This follows by differentiating $e^sG_s*p$ and integrating in $s$. For $p,H\ge0$, it yields the useful majorant $0\le p\le e^tG_t*p\le e^tP/(4\pi t)$ almost everywhere. These conventions and identities will be used for the normalized cap class in Section~\ref{sec:caps}.

\section{Hemisphere rank and the uniform one-node gap}\label{sec:rank}
The following lemma recovers, in the present spherical setting, M\"oller's odd-degree node bound~\cite{Moller}. Its hemisphere proof applies odd exactness to $(u\cdot x)p(x)^2$, with $p\in H_m$, and the reproducing kernel \eqref{eq:hom-kernel} gives the accompanying individual-weight estimate.

\begin{lemma}\label{lem:hemisphere}
Let $m\ge0$ be an integer, and let $X=\{x_1,\ldots,x_N\}\subset\Sph$ and $w_1,\ldots,w_N>0$ define a cubature formula exact through degree $2m+1$ with respect to the normalized surface measure $\sigma$. For every $u\in\Sph$ satisfying
\[
 u\cdot x_i\ne0,\qquad i=1,\ldots,N,
\]
the open hemispheres
\[
 H_+(u)=\{x\in\Sph:u\cdot x>0\},
 \qquad
 H_-(u)=\{x\in\Sph:u\cdot x<0\}
\]
satisfy
\[
 \#\bigl(X\cap H_+(u)\bigr)\ge h_m,
 \qquad
 \#\bigl(X\cap H_-(u)\bigr)\ge h_m,
 \qquad h_m=\frac{(m+1)(m+2)}2.
\]
Moreover,
\[
 w_i\le\frac1{F_m}=\frac1{(m+1)(m+2)},
 \qquad i=1,\ldots,N.
\]
\end{lemma}
\begin{proof}
Fix $u\in\Sph$ satisfying $u\cdot x_i\ne0$ for all $i$. For $p\in H_m$, the polynomial $(u\cdot x)p(x)^2$ is odd, so
\begin{equation}\label{eq:hemisphere}
 \sum_{u\cdot x_i>0}w_i(u\cdot x_i)p(x_i)^2
 =\sum_{u\cdot x_i<0}w_i(-u\cdot x_i)p(x_i)^2.
\end{equation}
Both quadratic forms on $H_m$ are positive definite. If one vanishes at $p$, positivity forces $p$ to vanish at all nodes in that hemisphere; the equality forces vanishing in the other hemisphere. Exactness for $p^2$ then gives $\int p^2\dd\sigma=0$. Each positive-definite form is a sum of rank-one evaluations, and therefore contains at least $h_m$ terms.

For the weight bound, fix a node $y$. The polynomial
\[
 (1+x\cdot y)\Kcal_m(y,x)^2
\]
is nonnegative, has degree at most $2m+1$, has integral $h_m$, and takes the value $2h_m^2$ at $y$. The integral assertion follows from reproduction and parity of the squared kernel. Exactness gives $2h_m^2w_y\le h_m$.
\end{proof}

\begin{lemma}\label{lem:interpolate}
Let $d\ge0$ be an integer, and let $A=\{[a_1],\ldots,[a_N]\}\subset\RP$ be a finite set with fixed representatives $a_i\in\R^3\setminus\{0\}$. Write $H_d=\R[X,Y,Z]_d$ for the space of homogeneous polynomials of degree $d$. If the evaluation map
\[
 \begin{aligned}
 \operatorname{ev}_{A,d}:H_d&\longrightarrow\R^N,\\
 p&\longmapsto\bigl(p(a_1),\ldots,p(a_N)\bigr)
 \end{aligned}
\]
is surjective, then, for every $z=[b]\in\RP\setminus A$ with $b\in\R^3\setminus\{0\}$, the evaluation map
\[
 \begin{aligned}
 \operatorname{ev}_{A\cup\{z\},d+1}:H_{d+1}&\longrightarrow\R^{N+1},\\
 p&\longmapsto\bigl(p(a_1),\ldots,p(a_N),p(b)\bigr)
 \end{aligned}
\]
is surjective. Surjectivity of either map is independent of the choice of representatives.
\end{lemma}
\begin{proof}
Multiplication by a linear form nonzero at every $a_i$ shows that $\operatorname{ev}_{A,d+1}$ is surjective. If the extended map were not surjective, evaluation at $b$ on $H_{d+1}$ would be a linear combination of the evaluations at $a_1,\ldots,a_N$:
\[
 r(b)=\sum_{j=1}^N\lambda_j r(a_j)
 \qquad(r\in H_{d+1}).
\]
For each $i$, choose $p_i\in H_d$ with $p_i(a_j)=\delta_{ij}$ and apply this identity to $r=\ell p_i$, where $\ell\in H_1$. This gives
\[
 \ell(b)p_i(b)=\lambda_i\ell(a_i)
 \qquad(\ell\in H_1),
\]
and hence $p_i(b)b=\lambda_i a_i$. Since $[b]\ne[a_i]$, we obtain $\lambda_i=0$ for every $i$. This would make every polynomial in $H_{d+1}$ vanish at $b$, a contradiction. Finally, rescaling the representatives multiplies the evaluation maps by invertible diagonal matrices on their codomains, so surjectivity does not depend on these choices.
\end{proof}

\begin{lemma}\label{lem:mixed-zero}
Let $m\ge1$ be an integer, and define
\[
 \mathcal B_m(p,q)=\int_{\Sph}Zpq\dd\sigma,
 \qquad p\in H_{m-1},\quad q\in H_m,
\]
and
\[
 W_m=\{q\in H_m:\mathcal B_m(p,q)=0\ \text{for every }p\in H_{m-1}\}.
\]
The pairing $\mathcal B_m$ has rank $h_{m-1}$, and $\dim W_m=m+1$. Moreover,
\[
 \begin{aligned}
 \mathcal Z_{\R}(W_m)
 &:=\{[a:b:c]\in\RP:q(a,b,c)=0\ \text{for every }q\in W_m\}\\
 &=\begin{cases}
 \varnothing, & m\ \text{even},\\[2mm]
 \{[0:0:1]\}, & m\ \text{odd}.
 \end{cases}
 \end{aligned}
\]
\end{lemma}
\begin{proof}
On $ZH_{m-1}\subset H_m$, the pairing becomes $\int Z^2pp'\dd\sigma$, a positive-definite inner product on $H_{m-1}$. Thus the rank and dimension assertions follow, and
\[
 W_m=(ZH_{m-1})^\perp
\]
in $H_m$ with its spherical $L^2$ inner product. If $[v]\in\mathcal Z_{\R}(W_m)$, choose its representative $v\in\Sph$. Then the reproducing representative $\Kcal_m(v,\cdot)$ belongs to $ZH_{m-1}$ and vanishes on the equator. Writing $\rho=(v_1^2+v_2^2)^{1/2}$ and rotating horizontally gives
\[
 C_m^{3/2}(\rho\cos\theta)=0\qquad(0\le\theta\le2\pi).
\]
A nonzero polynomial cannot vanish on an interval, so $\rho=0$. At the pole, the homogeneous representative of $C_m^{3/2}(Z)$ is divisible by $Z$ exactly when $m$ is odd. Indeed $C_m^{3/2}(0)$ is zero for odd $m$ and nonzero for even $m$. This proves the stated zero sets.
\end{proof}

\begin{theorem}\label{thm:gap}
For $m\ge2$, no positive formula of degree $2m+1$ on $\Sph$ has exactly $F_m+1$ nodes.
\end{theorem}
\begin{proof}
Assume $N=2h_m+1$. Lemma~\ref{lem:hemisphere} says that every open hemisphere whose boundary contains no node contains $h_m$ or $h_m+1$ nodes. Remove all antipodal pairs, denote their number by $k$, and call the remaining nodes $Y$. Every pair contributes one node to each such hemisphere, hence, for every $u\in\Sph$ satisfying $u\cdot x_i\ne0$ for all $i$,
\begin{equation}\label{eq:discrepancy}
 \sum_{y\in Y}\sgn(u\cdot y)\in\{-1,1\}.
\end{equation}
The projective directions of $Y$ are collinear. Otherwise, choose a projective line avoiding this finite set and use it as the line at infinity. The real Sylvester--Gallai theorem~\cite{BorweinMoser} then gives an affine, hence projective, line containing exactly two of the directions. Near a normal to their span, their two signs can be varied independently while every other sign stays fixed. The three resulting values $c-2,c,c+2$ cannot all lie in $\{-1,1\}$. Rotate so that the common line is $Z=0$.

Choose an open hemisphere containing exactly $h_m$ nodes and having no node on its boundary, and let $B$ be its projective directions. The hemisphere form is positive definite with $h_m$ summands, so evaluation on $B$ is an isomorphism on $H_m$. Put $g=h_{m-1}$. If $k<g$, there is a nonzero $Q\in H_{m-1}$ vanishing at every paired direction of $B$ off the equator. Then $ZQ$ vanishes on all of $B$, a contradiction. Thus $k\ge g$.

If $k>g$, the full projective support is obtained from $B$ by adjoining
\[
 s=h_m+1-k\le h_m-g=m+1
\]
points. Lemma~\ref{lem:interpolate} gives interpolation in degree $d=m+s\le2m+1$. If $d$ is even, raise it to $d+1$ by a linear form nonzero on the finite support, dividing the prescribed values by that form before interpolating. There is therefore an odd homogeneous polynomial of degree at most $2m+1$ which vanishes on every paired direction and equals one at every actual unpaired node. Its integral is zero, whereas its cubature sum is the positive total weight of $Y$. Consequently $k=g$.

None of these $g$ paired directions lies on $Z=0$, and their evaluations on $H_{m-1}$ are independent: either failure would again produce a nonzero $ZQ$ vanishing on $B$. Choose representatives $a_1,\ldots,a_g$ and total pair weights $\beta_j>0$. For $q\in W_m$, degree-$2m$ exactness gives
\[
 0=\sum_{j=1}^g\beta_jZ(a_j)q(a_j)p(a_j)
 \qquad(p\in H_{m-1}).
\]
The evaluation matrix is invertible and every $\beta_jZ(a_j)$ is nonzero. Hence all $q\in W_m$ vanish at all $g$ directions. Lemma~\ref{lem:mixed-zero} allows at most one such direction, whereas $g=m(m+1)/2\ge3$. This contradiction proves the theorem.
\end{proof}

\section{Low-degree rigidity and exact constructions}\label{sec:low}
Theorem~\ref{thm:gap} excludes twenty-one nodes in degree seven and thirty-one nodes in degree nine. It remains to exclude the respective Fisher values, twenty and thirty. At equality the evaluation matrix is square, and its orthogonality imposes an impossible integer count on the pairwise inner products.

\begin{lemma}\label{lem:tight}
If a degree-$2m+1$ positive formula has $N=F_m$, its support is antipodal. On one representative from each of its $h_m$ pairs, evaluation on $H_m$ is an isomorphism. The two weights in each pair are equal, their total is $1/h_m$, and distinct pair directions satisfy $\Kcal_m(v_i,v_j)=0$.
\end{lemma}
\begin{proof}
Every open hemisphere whose boundary contains no node has $h_m$ nodes. An unpaired node would change this count upon crossing its orthogonal great circle at a point on no other nodal great circle. Thus the support is antipodal. The positive hemisphere form has exactly $h_m$ summands, so evaluation is invertible. For a Lagrange form $p_j\in H_m$, the odd test $(u\cdot x)p_j(x)^2$ forces equality of the two weights in the $j$th pair. If $\lambda_j$ denotes the pair weight and $p_1,\ldots,p_{h_m}$ is an orthonormal basis, the square matrix $M_{j\alpha}=\sqrt{\lambda_j}p_\alpha(v_j)$ satisfies $M^TM=I$ by degree-$2m$ exactness. Thus $MM^T=I$. Its diagonal and off-diagonal entries, together with \eqref{eq:hom-kernel}, give the result.
\end{proof}

\begin{proposition}\label{prop:low-lower}
There are no positive degree-seven formulas with twenty nodes and no positive degree-nine formulas with thirty nodes. Consequently $N_7\ge22$ and $N_9\ge32$.
\end{proposition}
\begin{proof}
For $m=3$ the kernel is
\[
 \Kcal_3(t)=\frac52t(7t^2-3).
\]
Lemma~\ref{lem:tight} would give ten equally weighted projective directions with off-diagonal inner products $0$ or $\pm\sqrt{3/7}$. Degree-two exactness gives $\sum_jv_jv_j^T=(10/3)I_3$. The Gram matrix $G$ consequently has $\tr(G^2)=100/3$. If $e$ is the number of unordered nonorthogonal pairs, then
\[
 10+\frac67e=\frac{100}{3},\qquad e=\frac{245}{9},
\]
which is impossible.

For $m=4$,
\[
 \Kcal_4(t)=1+5P_2(t)+9P_4(t)
 =\frac{15}{8}(21t^4-14t^2+1).
\]
The fifteen hypothetical pair directions have total weights $1/15$ and squared off-diagonal inner products
\[
 \frac{7+2\sqrt7}{21},\qquad \frac{7-2\sqrt7}{21}.
\]
Now $\sum_jv_jv_j^T=5I_3$ and $\tr(G^2)=75$. If $a$ of the $105$ unordered pairs have the first squared value, then
\[
 15+2\left(a\frac{7+2\sqrt7}{21}+(105-a)\frac{7-2\sqrt7}{21}\right)=75,
 \qquad a=\frac{105}{2}-\frac{15\sqrt7}{4},
\]
another noninteger count. Theorem~\ref{thm:gap} excludes the next cardinalities.
\end{proof}

\subsection{The Reznick twenty-two-node formula}
The sextic identity from~\cite{Reznick}, also reproduced in~\cite{EW}, takes the form
\begin{equation}\label{eq:reznick}
\begin{split}
540(X^2+Y^2+Z^2)^3={}&378X^6+378Y^6+280Z^6\\
&+\sum_{s=\pm1}(\sqrt3X+2sZ)^6
 +\sum_{s=\pm1}(\sqrt3Y+2sZ)^6\\
&+\sum_{s,t=\pm1}(\sqrt3X+s\sqrt3Y+tZ)^6.
\end{split}
\end{equation}
Its terms give eleven projective directions
\begin{equation}\label{eq:22directions}
 e_1,e_2,e_3,\qquad
 \frac{(\sqrt3,0,\pm2)}{\sqrt7},\quad
 \frac{(0,\sqrt3,\pm2)}{\sqrt7},\quad
 \frac{(\sqrt3,\pm\sqrt3,\pm1)}{\sqrt7}.
\end{equation}
The two signs in the last expression are independent.

\begin{proposition}\label{prop:22}
Replace every direction in \eqref{eq:22directions} by its two antipodal representatives. Give the four nodes $\pm e_1,\pm e_2$ weight $1/20$, the two nodes $\pm e_3$ weight $1/27$, and the remaining sixteen nodes weight $49/1080$. The resulting positive formula has degree seven. Hence $N_7=22$.
\end{proposition}
\begin{proof}
Absorb the scalar coefficients in \eqref{eq:reznick} into the corresponding linear forms. The sum of the sixth powers of their coefficient-vector norms is
\[
 378+378+280+8\cdot343=3780.
\]
Dividing by this number gives a projective identity
\[
 \sum_{j=1}^{11}\lambda_j(u\cdot v_j)^6=\frac{\norm u^6}{7}
 =\int_{\Sph}(u\cdot x)^6\dd\sigma,
\]
where the projective weights are $1/10,1/10,2/27$ on the axes and $49/540$ on each remaining direction. Polarization yields exactness for all homogeneous sextics. The parity reduction in Subsection~\ref{subsec:spherical-cubature} then yields degree-seven exactness after equal splitting within each antipodal pair. The weights are positive, sum to one, and the nodes are distinct. Proposition~\ref{prop:low-lower} completes the proof.
\end{proof}

The bound $w_i\le1/20$ is attained at four nodes. More precisely, with
\[
 q(t)=\frac{t(7t^2-3)}4,\qquad F(t)=(1+t)q(t)^2,
\]
one has $F(1)=2$ and $\int_{\Sph}F(x\cdot y)\dd\sigma=1/10$. Thus every degree-seven formula satisfies
\begin{equation}\label{eq:weight-stability}
 \sum_{j\ne i}w_jF(x_i\cdot x_j)=2\left(\frac1{20}-w_i\right).
\end{equation}
Equality in the weight bound forces $x_i\cdot x_j\in\{-1,0,\pm\sqrt{3/7}\}$ for $j\ne i$. Thus saturation of a single weight restricts every inner product with that node.

\subsection{The Hughes--Waldron thirty-two-node formula}\label{subsec:nine}\label{subsec:nine-input}
The construction in~\cite{HW} gives a weighted spherical half-design of order eight on sixteen projective directions. Six directions come from the twelve vertices of a regular icosahedron and ten from the twenty vertices of its dual dodecahedron, in dual relative orientation. The normalization used there is $\widehat w_j=16\lambda_j$, where $\lambda_j$ is the probability weight of a projective direction. Hence
\[
\begin{aligned}
 \lambda_j&=\frac1{16}\frac{20}{21}=\frac5{84}
 &&\text{on the six icosahedral directions},\\
 \lambda_j&=\frac1{16}\frac{36}{35}=\frac9{140}
 &&\text{on the ten dodecahedral directions}.
\end{aligned}
\]
As shown in~\cite{HW}, splitting each projective weight equally between the two unit antipodes gives a degree-nine cubature formula. Its individual sphere weights are therefore
\begin{equation}\label{eq:32weights}
 \frac5{168}\quad\text{at each icosahedral vertex},\qquad
 \frac9{280}\quad\text{at each dodecahedral vertex}.
\end{equation}
They are positive and satisfy $12(5/168)+20(9/280)=1$. Exactness, rather than only this normalization, is supplied by the cited construction.

\begin{corollary}\label{cor:nine}\label{cor:nine-conditional}
One has $N_9=32$.
\end{corollary}
\begin{proof}
The preceding formula gives $N_9\le32$, and Proposition~\ref{prop:low-lower} gives the reverse inequality.
\end{proof}
The general construction in Section~\ref{sec:cascade} gives thirty-four nodes in degree nine. The specialized formula above improves this finite-degree count.

\section{Lobatto constructions and the sharp node-count bound}\label{sec:cascade}
Building on the classical spherical product-Lobatto framework~\cite{StroudSphere,LuoMeng} and related moment-map constructions~\cite{KuperbergHatBox}, we combine the Gauss--Lobatto rule with the Toeplitz moment tools of Subsection~\ref{subsec:lobatto-toeplitz}. The latitude levels and total weights are fixed by the one-dimensional rule; the step specific to the present construction is rank-controlled positive angular-moment compression through the nested scalar cascade below.

\subsection{The antipodal latitude scheme and the baseline formula}\label{subsec:lobatto-baseline}
Throughout this section, let $m\ge3$, put $d=2m$ and $\epsilon=m\bmod2$, and use the nodes $z_j$ and weights $b_j$ of Lemma~\ref{lem:lobatto}. Pair the nonzero interior latitude measures antipodally and, when $\epsilon=1$, take an antipodal equatorial measure. The angular moments must satisfy \eqref{eq:latitude-moment-system}; the equation count for each frequency is given by \eqref{eq:q-h}.

As a baseline, put a regular $(2m+1)$-gon with equal angular weights on each nonzero interior latitude and, when necessary, a regular $(2m+2)$-gon on the equator. Choose the polygons at opposite latitudes as antipodal images. Their nonzero Fourier moments through degree $d$ vanish, so the latitude criterion proves exactness through degree $2m+1$. There are $m-\epsilon$ nonzero interior latitudes, $\epsilon$ equatorial circles, and two poles. Hence
\begin{equation}\label{eq:baseline-upper}
 N_{2m+1}\le B_m:=(m-\epsilon)(2m+1)+\epsilon(2m+2)+2
 =2m^2+m+\epsilon+2.
\end{equation}
The following construction decreases this count by allowing selected higher angular moments to be nonzero while keeping the latitude equations satisfied.

\subsection{The scalar cascade and its node budget}
Put $x=m-\epsilon$, so that there are $x/2$ positive latitude pairs. For now use the baseline equator when $\epsilon=1$. Choose
\[
 1\le r\le x/2,\qquad L=d+1-2r,\qquad p=2r.
\]
Then $L>d/2$. Set the moments below $L$ to zero and use the block representation \eqref{eq:moment-block}, with $p=d+1-L=2r$. It remains to choose the lower-triangular Toeplitz contraction $C$ so as to make $\dim\ker(I-C^*C)$ large on most latitude pairs.

First prescribe only $c_L=\beta_{-1}$ on each positive latitude. The $r$ complex latitude equations and Lemma~\ref{lem:disks} leave at most $r$ pairs with $|\beta_{-1}|<1$; we call these pairs active. At every other pair, $C=\beta_{-1}I_p$ has defect-kernel dimension $p$, so \eqref{eq:moment-block} and Lemma~\ref{lem:toeplitz-atoms} give at most $L$ angular nodes.

For each still-active latitude, select an integer chain
\[
 r=h_0>h_1>\cdots>h_s=1,\qquad h_{j+1}\le\lfloor h_j/2\rfloor,
\]
and prescribe
\begin{align}
 C&=\beta_{-1}I_{2r}+(1-|\beta_{-1}|^2)E_{2r,h_0}(B_0),\label{eq:nested-C}\\
 B_j&=\beta_jI_{h_j}+(1-|\beta_j|^2)E_{h_j,h_{j+1}}(B_{j+1}),
 \qquad B_s=\beta_s I_1.\label{eq:nested-B}
\end{align}
At each level the preceding variables are held fixed. The new frequency is $d+1-h_j$, and its moment is $\beta_j$ times the product of the positive defect factors from the earlier active levels. Its latitude constraints are therefore still homogeneous complex linear equations, with $\beta_j=0$ feasible. Lemma~\ref{lem:disks} leaves at most $q(h_j)$ pairs active after this choice.

The frequencies increase with $j$, so a new choice does not alter any previously imposed moment. If $|\beta_j|=1$, the later moments vanish. Repeated application of Lemma~\ref{lem:corner} then gives final defect-kernel dimension $h_j$ at that latitude. This is the link between the number of active pairs and the atom count in the following estimate.

The reference count, with all nonzero interior circles using $L$ atoms, is $B_m-2xr$. The first transition from defect $2r$ to defect $r$ costs at most $2r^2$ nodes. Each later transition costs at most $2q(h_j)(h_j-h_{j+1})$. At most one latitude pair remains after the final scalar stage, costing at most two additional nodes. Thus
\begin{equation}\label{eq:chain-bound}
 N_{2m+1}\le B_m-2xr+Q(h_0,\ldots,h_s),
\end{equation}
where
\begin{equation}\label{eq:chain-Q}
 Q(h_0,\ldots,h_s)=2r^2+
 2\sum_{j=0}^{s-1}\left\lceil\frac{h_j}{2}\right\rceil(h_j-h_{j+1})+2.
\end{equation}
Lemma~\ref{lem:corner} ensures positivity of every moment matrix. Lemma~\ref{lem:toeplitz-atoms} realizes its moments with positive angular weights, which are then multiplied by the positive Lobatto weights. The resulting measure satisfies \eqref{eq:latitude-moment-system}, so the latitude criterion in Subsection~\ref{subsec:lobatto-toeplitz} proves that it is a cubature formula.

For the halving chain $h_{j+1}=\lfloor h_j/2\rfloor$, define
\begin{equation}\label{eq:E-recursion}
 E(0)=0,\quad E(1)=\frac43,\quad
 E(2v)=E(v),\quad E(2v+1)=E(v)+\frac43(v+1).
\end{equation}
The odd recurrence is valid also at $v=0$. Substitution in \eqref{eq:chain-Q} gives $Q_b(r)=\frac83r^2+E(r)$, and induction gives
\begin{equation}\label{eq:E-basic}
 \frac43\le E(r)\le\frac43r\quad(r\ge1),\qquad E(2^ar)=E(r).
\end{equation}
Consequently
\begin{equation}\label{eq:even-budget}
 N_{2m+1}\le B_m-\frac38x^2+
 \min_{0\le r\le\lfloor x/2\rfloor}
 \left\{\frac83(r-3x/8)^2+E(r)\right\}.
\end{equation}
The term $r=0$ simply denotes the baseline construction. For $m=4$, taking $r=1$ or $r=2$ gives $N_9\le34$.

\subsection{Equatorial coupling for odd \texorpdfstring{$m$}{m}}\label{subsec:equator}
For odd $m$, the equatorial moment can be chosen together with those of the latitude pairs. Its contribution to the active-coordinate count compensates exactly for the difference in equatorial node cost. Suppose $m\ge3$ is odd and put $x=m-1$. Define the tail cost
\begin{equation}\label{eq:T-recursion}
 T(0)=0,\qquad T(h)=2\left\lceil\frac h2\right\rceil
 \left(h-\left\lfloor\frac h2\right\rfloor\right)+T(\lfloor h/2\rfloor)
 =\frac23h^2+E(h).
\end{equation}
If $A$ latitude pairs are active at a parent defect $p$ and the next defect is $h=\lfloor p/2\rfloor$, the subsequent cost is at most $2A(p-h)+T(h)$.

Choose an odd parent size $p=2r+1$ with $0\le r\le(m-3)/2$, so
\[
 L=d+1-p=2(m-r)>d/2.
\]
The first frequency now has $r+1$ complex equations. Include the equator as one further disk coordinate $\zeta$; it appears only in the constant latitude equation and adds no new equation.

To realize this coordinate antipodally, use $\psi=2\phi$ on the equator. In the compressed degree-$m$ circle moment problem prescribe only the moment $b_{m-r}=\zeta$. Since $m-r>m/2$, a saturated coordinate needs at most $m-r$ compressed atoms, hence $L$ sphere nodes. An interior coordinate needs at most $m+1$ compressed atoms, hence $d+2$ sphere nodes. Its other moments through the required degree remain zero, so an interior equator need not enter the later cascade. Equatorial saturation saves
\[
 (d+2)-L=p+1
\]
nodes.

At an extreme point of the joint disk system, there are two cases. If the equator is saturated, at most $r+1$ latitude pairs remain active. Their tail cost is at most
\[
 2(r+1)(p-r)+T(r)=2(r+1)^2+T(r)=T(p),
\]
and the equator saves $p+1$. If the equator is interior, at most $r$ latitude pairs remain active; their tail cost is
\[
 2r(p-r)+T(r)=T(p)-(p+1),
\]
and there is no equatorial saving. The same upper bound follows in both cases:
\begin{align}
 N_{2m+1}&\le B_m-xp+T(p)-(p+1)\notag\\
 & =2m^2+4-2(m-2)r+\frac83r^2+E(r).
 \label{eq:odd-budget}
\end{align}
This includes $r=0$, where no later cascade is needed. In particular $m=3$ gives $N_7\le22$.

Define, for every positive integer $y$,
\begin{equation}\label{eq:remainder-def}
 \mathfrak r(y)=\min_{0\le r\le\lfloor y/2\rfloor}
 \left\{\frac83(r-3y/8)^2+E(r)\right\}.
\end{equation}
Equations~\eqref{eq:even-budget} and \eqref{eq:odd-budget} yield the stronger, finite-integer form of our upper bound:
\begin{equation}\label{eq:exact-budget}
 N_{2m+1}\le
 \begin{cases}
 \frac{13}{8}m^2+m+2+\mathfrak r(m),&m\text{ even},\\[2pt]
 \frac{13}{8}m^2+\frac32m+\frac52+\mathfrak r(m-2),&m\text{ odd}.
 \end{cases}
\end{equation}
The expression comes from an antipodal construction and is an even integer when evaluated through its original budget. Any real upper envelope may therefore be rounded down to the nearest even integer.

\subsection{The binary remainder and its sharp constant}\label{sec:integer}
To estimate \eqref{eq:exact-budget}, we must minimize a quadratic rounding cost together with the binary defect $E(r)$. The neighboring values of $E$ are correlated, and this correlation determines the constant in the $y^{2/3}$ bound.

\begin{proposition}\label{prop:halving-optimal}
Among the chains allowed in \eqref{eq:chain-Q}, repeated downward halving minimizes that scalar cost. Moreover
\begin{equation}\label{eq:odd-sum}
 T(r)-T(r-1)=2\odd(r),\qquad
 T(r)=2\sum_{j=1}^r\odd(j),
\end{equation}
where $\odd(j)=j/2^{v_2(j)}$.
\end{proposition}
\begin{proof}
The recurrences in \eqref{eq:T-recursion} give the first identity by induction: an even difference reduces to the difference at $r/2$, and an odd difference equals $2r$. Summing proves the second. If the next child of $r$ is $h\le\lfloor r/2\rfloor$, its optimal remaining cost, by induction, is
\[
 2\lceil r/2\rceil(r-h)+T(h).
\]
Increasing $h$ by one changes this quantity by
\[
 -2\lceil r/2\rceil+2\odd(h+1)\le
 -2\lceil r/2\rceil+2(h+1)\le0.
\]
Thus the largest allowed child is an optimal choice at every stage. This does not assert uniqueness of the minimizing chain.
\end{proof}

\begin{lemma}\label{lem:E-correlation}
For every $v\ge1$,
\begin{equation}\label{eq:E-correlation}
 2E(v)+E(v+1)\le\frac83v+\frac43.
\end{equation}
\end{lemma}
\begin{proof}
Set $D(v)=3E(v)/4$, so $D(v)\le v$. The assertion is $2D(v)+D(v+1)\le2v+1$, with equality at $v=1$. If $v=2w$, its left side is $3D(w)+w+1\le4w+1$. If $v=2w+1$, it is $2D(w)+D(w+1)+2w+2\le4w+3$ by induction.
\end{proof}

\begin{proposition}\label{prop:remainder-envelope}
For every positive integer $y$,
\begin{equation}\label{eq:remainder-envelope}
 \mathfrak r(y)\le\min\left\{\frac76y^{2/3},\
 \gamma y^{2/3}+\frac43\right\}=\mathcal R(y).
\end{equation}
\end{proposition}
\begin{proof}
For the first branch, assume $y\ge8$ and choose the largest dyadic integer $q$ satisfying $q^3\le y$. Let $k$ be a nearest integer to $3y/(8q)-1/(4q^2)$, with ties rounded upwards, and put $r=qk$. Since $q\le y/4$, this choice has $1\le r\le y/2$. Using $E(qk)=E(k)\le4k/3$ and completing the square gives
\begin{equation}\label{eq:dyadic-rounding}
 \frac83(r-3y/8)^2+E(r)
 \le\frac{y}{2q}+\frac23q^2-\frac1{6q^2}.
\end{equation}
With $z=y^{1/3}/q\in[1,2)$, the identity
\[
 3z^3-7z^2+4=(z-1)(z-2)(3z+2)\le0
\]
shows that the first two terms on the right are at most $7y^{2/3}/6$. For $y=1,\ldots,7$, take respectively $r=0,1,1,2,2,2,2$. The residuals are $3/8,3/2,11/8,2,11/8,3/2,19/8$, all below this branch.

For the second branch, choose the largest dyadic $q$ with $(16/9)q^3\le y$, and write
\[
 \lambda=y/q^3\in[16/9,128/9),\qquad
 \frac{3y}{8q}=2v+u,\quad0\le u<2.
\]
Initially suppose that $v\ge1$ and all three candidates
\[
 r_0=2qv,\quad r_1=q(2v+1),\quad r_2=q(2v+2)
\]
are feasible. Put $z=E(v)/(4v/3)\in[0,1]$. The recurrence and Lemma~\ref{lem:E-correlation} give
\begin{align*}
 E(r_0)&\le z\frac{y}{4q},\\
 E(r_1)&\le(z+1)\frac{y}{4q}+\frac43,\\
 E(r_2)&\le\min(1,2-2z)\frac{y}{4q}+\frac43.
\end{align*}
The best candidate therefore has residual at most $q^2$ times
\begin{equation}\label{eq:three-parabolas}
 \min\left\{\frac83u^2+\frac{z\lambda}{4},\
 \frac83(u-1)^2+\frac{(z+1)\lambda}{4},\
 \frac83(u-2)^2+\frac{\min(1,2-2z)\lambda}{4}\right\}
\end{equation}
plus $4/3$.

For $z\le1/2$, increasing $z$ to $1/2$ cannot decrease the minimum in \eqref{eq:three-parabolas}. It therefore suffices to consider $z\in[1/2,1]$. The adjacent intersections are
\[
 u_{01}=\frac12+\frac{3\lambda}{64},\qquad
 u_{12}=\frac32+\frac{3(1-3z)\lambda}{64}.
\]
The middle parabola is present in the lower envelope exactly when $z\lambda\le64/9$. Each piece is convex, so a maximum occurs at an endpoint or an intersection. These comparisons give the bound
\begin{equation}\label{eq:H-envelope}
 H(\lambda)=
 \begin{cases}
 P(\lambda)=\frac23+\frac{3\lambda}{8}+\frac{3\lambda^2}{512},
 &16/9\le\lambda\le64/9,\\[2pt]
 V(\lambda)=\frac{22}{9}+\frac\lambda8+\frac{3\lambda^2}{512},
 &64/9\le\lambda\le128/9.
 \end{cases}
\end{equation}
Indeed, in the first range the first intersection is largest at $z=1$ and equals $P$. The second intersection is convex in $z$, so it is enough to compare the endpoint values
\[
 P_{1/2}=\frac23+\frac{5\lambda}{16}+\frac{3\lambda^2}{2048},\qquad
 P_1=\frac23+\frac\lambda4+\frac{3\lambda^2}{128}.
\]
They are bounded by $P$ because
\[
 P-P_{1/2}=\frac{\lambda(128+9\lambda)}{2048},\qquad
 P-P_1=\frac{\lambda(64-9\lambda)}{512}.
\]
In the second range, split at $z_*=64/(9\lambda)$. In the three-parabola regime the relevant endpoint values are $V$ and $P_{1/2}$. In the two-parabola regime they are $V$ and $Q_1=8/3+\lambda/8+3\lambda^2/2048$. The comparisons are
\begin{align*}
 V-P_{1/2}&=\frac{(9\lambda-256)(9\lambda-128)}{18432}\ge0,\\
 V-Q_1&=\frac{(9\lambda-64)(9\lambda+64)}{18432}\ge0.
\end{align*}
Spatial endpoint values are at most $\lambda/4$, and
\[
 V-\lambda/4=\frac3{512}(\lambda-32/3)^2+\frac{16}{9}>0.
\]
This proves \eqref{eq:H-envelope}.

To normalize the bound, compute
\[
 \lambda P'-\frac23P=-\frac49+\frac\lambda8+\frac{\lambda^2}{128},\qquad
 \lambda V'-\frac23V=-\frac{44}{27}+\frac\lambda{24}+\frac{\lambda^2}{128}.
\]
Both expressions increase on the positive axis. Each normalized piece $H(\lambda)/\lambda^{2/3}$ has its maximum at an endpoint. The middle endpoint has value
\[
 \frac{98/27}{(64/9)^{2/3}}=\frac{49\sqrt[3]3}{72}=\gamma.
\]
The two outside endpoint values are equal, and their ratio to this value is $73\sqrt[3]2/98<1$. Thus $H(\lambda)\le\gamma\lambda^{2/3}$, and $q^2\lambda^{2/3}=y^{2/3}$ proves the desired estimate whenever the three candidates are feasible.

For $y\ge32$, $q\le y/16$, so they all lie in $[1,y/2]$. For $2\le y<32$, the only exceptional values, with suitable replacement candidates, are
\[
\begin{array}{c|rrrrrrrrrr}
 y&2&3&4&5&6&7&11&15&22&23\\
 r&1&1&2&2&2&2&4&6&8&8
\end{array}
\]
with residuals $3/2,11/8,2,11/8,3/2,19/8,11/8,35/8,3/2,19/8$. Each is below $\gamma y^{2/3}+4/3$, using $\gamma>3/4$ and elementary cubing. For $y=1$, the choice $r=0$ gives $\mathfrak r(1)=3/8$.
\end{proof}

\begin{corollary}\label{cor:upper}
For every $m\ge3$, $N_{2m+1}\le2\lfloor\mathcal U_m/2\rfloor$.
\end{corollary}
\begin{proof}
Insert Proposition~\ref{prop:remainder-envelope} in \eqref{eq:exact-budget} and use the even cardinality of the constructed antipodal formula.
\end{proof}

\begin{theorem}\label{thm:sharp-remainder}
For the scalar budget \eqref{eq:remainder-def}, both parity subsequences satisfy
\begin{equation}\label{eq:sharp-remainder}
 \limsup_{\substack{y\to\infty\\y\ \mathrm{even}}}
 \frac{\mathfrak r(y)}{y^{2/3}}
 =\limsup_{\substack{y\to\infty\\y\ \mathrm{odd}}}
 \frac{\mathfrak r(y)}{y^{2/3}}=\gamma.
\end{equation}
\end{theorem}
\begin{proof}
The upper bound follows from Proposition~\ref{prop:remainder-envelope}. To attain it asymptotically, let $Q=64^j$ and define
\[
 A_Q=\frac{4Q-1}{3},\qquad v_Q=QA_Q-1,\qquad
 y_Q=\frac{16Qv_Q}{3}+\frac{20Q-2}{9}.
\]
Since $Q\equiv1\pmod9$, $v_Q$ is divisible by three and $y_Q$ is a positive even integer. Moreover
\begin{equation}\label{eq:sharp-sequence}
 \frac{y_Q}{Q^3}\longrightarrow\frac{64}{9},\qquad
 t_Q=\frac{3y_Q}{8}=2Qv_Q+\frac{5Q}{6}-\frac1{12}.
\end{equation}
For $D(r)=3E(r)/4$, repeatedly removing trailing binary ones gives
\[
 D(QA-1)=A(Q-1)+D(A-1).
\]
Together with $0\le D(A-1)\le A-1$, this implies
\[
 \frac{E(v_Q)}{Q^2}\to\frac{16}{9},\qquad
 \frac{E(v_Q+1)}{Q^2}\to0,\qquad
 \frac{E(2v_Q+1)}{Q^2}\to\frac{32}{9}.
\]
The three feasible candidates $2Qv_Q$, $Q(2v_Q+1)$, and $Q(2v_Q+2)$ have residuals divided by $Q^2$ tending respectively to
\[
 \frac83\left(\frac56\right)^2+\frac{16}{9},\qquad
 \frac83\left(\frac16\right)^2+\frac{32}{9},\qquad
 \frac83\left(\frac76\right)^2,
\]
all equal to $98/27$.

It remains to exclude better integers. Every minimizing $r_Q$ lies within $2Q$ of $t_Q$ for large $Q$, since otherwise the quadratic term alone exceeds these candidate values. Uniformly in that interval, $r_Q/Q^3\to8/3$. For odd $w\ge3$, the recurrence yields $E(w)\ge2(w+1)/3$.

If $Q$ divides $r_Q$, write $r_Q=Q(2v_Q+l)$. Only finitely many integers $l$ are relevant. Apart from $l=0,1,2$, the distance term alone has a limit larger than $98/27$; for example $l=-1$ gives $242/27$. The remaining three cases were just evaluated. If $Q/2$ divides $r_Q$ but $Q$ does not, its odd part is $2r_Q/Q$, so $\liminf E(r_Q)/Q^2\ge32/9$. The half-integer lattice is at distance at least $1/3$ from the limiting phase $5/6$, giving total limit inferior at least $104/27$. Finally, if $Q/2$ does not divide $r_Q$, its maximal dyadic factor is at most $Q/4$, and the odd-part bound gives $\liminf E(r_Q)/Q^2\ge64/9$.

Thus $\mathfrak r(y_Q)/Q^2\to98/27$. Dividing by the limit in \eqref{eq:sharp-sequence} proves the even subsequence assertion. Replacing $y_Q$ by $y_Q+1$ changes the center by only $3/8$, so the same $Q$-scale argument proves the odd subsequence assertion.
\end{proof}

The quadratic coefficient is obtained by minimizing $2-2\rho+\frac83\rho^2$ at $\rho=3/8$. Theorem~\ref{thm:sharp-remainder} concerns the remainder of this scalar budget. That budget charges the largest permitted number of active pairs at each stage, and those charges need not be attained simultaneously by a cubature formula. Its sharpness is therefore not a lower bound for $N_{2m+1}$.

\section{Weighted LP--Tur\'an bounds and cap transfer}\label{sec:cap-method}
We apply the LP and distance-graph estimates from Subsection~\ref{subsec:lp-turan-preliminaries}, with the radial conventions of Subsection~\ref{subsec:radial-preliminaries}. Yudin's cap gives the classical bound. Retaining its positive off-diagonal contributions leads to the continuous LP--Tur\'an inequality, the planar cap optimization, and the finite-degree transfer.

\subsection{The classical cap bound}\label{sec:yudin}
Yudin's cap gives the classical lower bound for spherical designs~\cite{Yudin}. Positive-weight spherical quadrature regularity and related cap estimates were developed by Reimer and Leopardi~\cite{ReimerHyper,Leopardi}. Applied to Yudin's cap kernel, Lemma~\ref{lem:weighted-lp} gives a short weighted derivation in the present normalization and excludes the Fisher value in the remaining degrees $2m+1$ with $m\ge5$.

\begin{proposition}\label{prop:yudin}
For $m\ge1$, with $a_m$ as in Theorem~\ref{thm:uniform-main},
\begin{equation}\label{eq:yudin}
 N_{2m+1}\ge\frac2{1-a_m}>
 \frac{160m^2+400m+220}{147}.
\end{equation}
In particular $N_{2m+1}\ge F_m+2$ for all $m\ge3$.
\end{proposition}
\begin{proof}
Put $n=2m+2$, $a=a_m$, and $b=-P_n(a)>0$. Define
\[
 f(u)=(P_n(u)+b)\ind_{[a,1]}(u),\qquad g(u)=\ind_{[a,1]}(u).
\]
Both functions are nonnegative. The rightmost critical point is a minimum, and $P_n$ is increasing from $a$ to $1$. Since $f(a)=f'(a)=0$, its zero extension introduces no boundary term in the Legendre equation. For $\lambda_\ell=\ell(\ell+1)$ and Legendre coefficients $f_\ell,g_\ell$,
\[
 (\lambda_n-\lambda_\ell)f_\ell=\lambda_n b g_\ell.
\]
By \eqref{eq:zonal-convolution}, $G=f*g$ is nonnegative and has coefficients $f_\ell g_\ell/(2\ell+1)$, nonpositive for $\ell>n$ and zero at $\ell=n$. Absolute uniform convergence follows from $f,g\in L^2$, as in Subsection~\ref{subsec:lp-turan-preliminaries}. Moreover $G(1)=f_0$ and $G_0=f_0g_0$, where $g_0=(1-a)/2$. Lemma~\ref{lem:weighted-lp} proves the first inequality.

To estimate $a$, set $\kappa=2m+5/2$, $\theta_m=\arccos a_m$, and
\[
 \nu(\theta)=(\sin\theta)^{3/2}P'_n(\cos\theta).
\]
The transformed equation is
\[
 \nu''+\left(\kappa^2-\frac3{4\sin^2\theta}\right)\nu=0.
\]
Compare its regular solution at zero with $v(\theta)=\sqrt\theta J_1(k\theta)$, where $k^2=\kappa^2-3/4$. On $(0,\pi/2)$, $\csc^2\theta\le1+\theta^{-2}$, so Sturm comparison yields
\[
 \theta_m\le\frac{R_0}{\sqrt{\kappa^2-3/4}},\qquad
 \frac2{1-a_m}\ge\frac{4(\kappa^2-3/4)}{R_0^2}.
\]
To justify comparison at the singular endpoint, match the positive $\theta^{3/2}$ leading terms. The Wronskian $\nu'v-\nu v'$ tends to zero at the origin and has nonpositive derivative while both solutions are positive. If $\nu$ had no zero by the first zero of $v$, the nonincreasing ratio $\nu/v$ would tend to positive infinity, a contradiction. The estimate below places the comparison zero in $(0,\pi/2)$ for every $m\ge1$.

At $x=\sqrt{147/10}$ the Bessel series gives
\[
 \frac{2J_1(x)}x
 <\sum_{r=0}^6\frac{(-147/40)^r}{r!(r+1)!}
 =-\frac{11838750593}{26214400000000}<0.
\]
The remaining alternating tail starts with a negative term and has decreasing absolute values. Since $J_1$ is positive just to the right of zero, $R_0^2<147/10<16$. Substitution proves the strict rational estimate in \eqref{eq:yudin}. For a direct integer bound, the strict inequality gives $N_{2m+1}\ge\lfloor(160m^2+400m+220)/147\rfloor+1$.

For $m\ge5$ the rational expression is larger than $F_m$, because the numerator of its difference from $F_m$ is $13m^2-41m-74$, positive at $m=5$ and increasing thereafter. The cases $m=3,4$ were excluded in Proposition~\ref{prop:low-lower}. Combining with Theorem~\ref{thm:gap} proves \eqref{eq:Fishergap}.
\end{proof}

Using $P'_{2m+2}=(2m+3)P_{2m+1}^{(1,1)}/2$, the Mehler--Heine formula for Jacobi polynomials~\cite{DLMF} gives
\[
 1-a_m\sim\frac{R_0^2}{2(2m+1)^2},\qquad
 \liminf_{m\to\infty}\frac{N_{2m+1}}{m^2}\ge\frac{16}{R_0^2}.
\]

\subsection{Continuous weighted LP--Tur\'an inequalities}\label{sec:lp-turan}
The diagonal estimate of Lemma~\ref{lem:weighted-lp} discards positive off-diagonal values of the kernel. We retain them by applying \eqref{eq:threshold-integral} to a nonincreasing lower envelope. Throughout this subsection, $F$, $f=1/F$, and $\kappa_0$ are the packing quantities defined in Subsection~\ref{subsec:lp-turan-preliminaries}; the kernel itself need not be monotone.

\begin{proposition}\label{prop:continuous-sphere}
Let $G$ satisfy the hypotheses of Lemma~\ref{lem:weighted-lp}, with angular support in $[0,\Theta]$, where $\Theta\le1/2$. Put $g_0=G_0$ and $A_0=G(1)$. Suppose $\ell$ is continuous, nonnegative and nonincreasing, satisfies $\ell(\theta)\le G(\cos\theta)$, $\ell(0)=c<A_0$, and $\ell(\Theta)=0$. Then
\begin{equation}\label{eq:continuous-sphere}
 N_t\ge\frac{A_0-c}{g_0-\int_0^\Theta f'(\theta)\ell(\theta)\dd\theta},
\end{equation}
and its denominator is at least $(1-\kappa_0)g_0>0$.
\end{proposition}
\begin{proof}
Apply \eqref{eq:threshold-integral} to the distance graphs of the cubature nodes. Retaining the diagonal difference $A_0-c$ separately and using the spectral upper estimate \eqref{eq:lp-energy} gives
\[
 g_0\ge\sum_{i,j}w_iw_jG(x_i\cdot x_j)
 \ge(A_0-c)\sum_iw_i^2+\int_0^\Theta f(\theta)(-d\ell(\theta)).
\]
The last integral equals $\int f'\ell$ by integration by parts. Since $\sum_iw_i^2\ge1/N$, this proves the ratio once the denominator is positive. Independently, \eqref{eq:packing-derivative} and $\ell\le G$ imply
\[
 \int_0^\Theta f'\ell\le\kappa_0\frac12
 \int_0^\Theta G(\cos\theta)\sin\theta\dd\theta=\kappa_0g_0.
\]
This establishes positivity without relying on the ratio itself.
\end{proof}

For finitely many thresholds $\theta_1<\cdots<\theta_J$ and decreasing levels $g_1\ge\cdots\ge g_J\ge g_{J+1}=0$, the same argument gives
\begin{equation}\label{eq:finite-threshold}
 N_t\ge\frac{G(1)-g_1}
 {G_0-\sum_{j=1}^J(g_j-g_{j+1})/F(\theta_j)}.
\end{equation}
The continuous version retains all admissible levels instead of selecting a finite staircase.

\subsection{The planar functional and complete height optimization}
Use the planar conventions of Subsection~\ref{subsec:radial-preliminaries}. Let $p\ge0$ be a radial cap supported in $\overline D_R\subset\R^2$ with positive companion $H=(1+\Delta)p$, a finite measure. Assume $p\in H^1$ and write
\[
 P=\int p=H(\R^2)>0,\qquad K=p*H,\quad a=K(0),\quad S=2R.
\]
Assume $K$ is continuous and nonnegative, with $a>0$. The normalized class in Subsection~\ref{sec:caps} satisfies the first two properties; the case $a=0$ will be assigned coefficient zero. Define
\begin{equation}\label{eq:profile-def}
 \rho(s)=\min_{0\le u\le s}K(u),\qquad
 d(s)=\frac{\sqrt3}{8\pi}s^2,\qquad
 \ell_c(s)=\min\{c,\rho(s)\},\quad0\le c\le a,
\end{equation}
and
\begin{equation}\label{eq:planar-functional}
 M_0=\frac{P^2}{4\pi},\quad
 J(c)=\int_0^S d'(s)\ell_c(s)\dd s,\quad
 D(c)=M_0-J(c),\quad
 C(p)=\max_{0\le c\le a}\frac{4(a-c)}{D(c)}.
\end{equation}
The running minimum makes $\ell_c$ the greatest nonincreasing function below $K$ with initial value $c$. Indeed, any such function must be at most both $c$ and every earlier value of $K$.

The mass identity $\int K=P^2$ gives a uniform denominator gap:
\begin{equation}\label{eq:mass-gap}
 J(c)\le\frac{\sqrt3}{4\pi}\int_0^S sK(s)\dd s
 =\frac{\sqrt3}{8\pi^2}P^2,
 \qquad D(c)\ge\left(1-\frac{\sqrt3}{2\pi}\right)M_0>0.
\end{equation}

\begin{proposition}\label{prop:height-root}
Assume in addition that $K(s)>0$ for $0\le s<S$ and $K(S)=0$. If $d(S)a\le M_0$, the optimal height is zero and $C(p)=4a/M_0$. Otherwise there is a unique $s_p\in(0,S)$ satisfying
\begin{equation}\label{eq:height-root}
 M_0=d(s_p)a+\int_{s_p}^S d'(u)\rho(u)\dd u,
\end{equation}
and $c_p=\rho(s_p)\in(0,a)$ is an optimal height, with
\begin{equation}\label{eq:root-value}
 C(p)=\frac4{d(s_p)}=\frac{32\pi}{\sqrt3\,s_p^2}.
\end{equation}
The conclusion does not require that $K$ be monotone or that $\rho$ have no plateaus.
\end{proposition}
\begin{proof}
Set $\Psi(s)=d(s)a+\int_s^S d'(u)\rho(u)\dd u$. By \eqref{eq:mass-gap}, $\Psi(0)=J(a)<M_0$, whereas $\Psi(S)=d(S)a$. Also
\[
 \Psi'(s)=d'(s)(a-\rho(s))\ge0,
\]
and it is strictly positive after $\rho$ first becomes smaller than $a$. Thus the root in the second case is unique.

In the first case, $J(c)\le d(S)c$ gives $D(c)\ge M_0(1-c/a)$, proving optimality of $c=0$. In the second, monotonicity of $\rho$ gives $D(c_p)=d(s_p)(a-c_p)$. For every other height,
\begin{align*}
 D(c)-d(s_p)(a-c)
 ={}&\int_0^{s_p}d'(u)\big(c-\min(c,\rho(u))\big)\dd u\\
 &+\int_{s_p}^Sd'(u)\big(\rho(u)-\min(c,\rho(u))\big)\dd u\ge0.
\end{align*}
This proves global optimality and \eqref{eq:root-value} without differentiating the running minimum.
\end{proof}

The corresponding finite spherical optimization is the same argument with $d,\rho,M_0,a$ replaced by $f,\rho_G,g_0,A_0$, where $\rho_G(\theta)=\min_{u\le\theta}G(\cos u)$. In its positive-height branch the optimal lower bound is $1/f(\theta_G)$, with $\theta_G$ solving
\[
 g_0=f(\theta_G)A_0+\int_{\theta_G}^{\Theta}f'(u)\rho_G(u)\dd u.
\]
If the optimum occurs at height zero, the bound is instead $A_0/g_0$.

Under the positivity assumptions of Proposition~\ref{prop:height-root}, continuous thresholds strictly improve every fixed nonzero finite staircase with top height below $a$. If the last threshold is $s_J<S$ and its top height is $c_1>0$, the full envelope $\ell_{c_1}$ is positive on $(s_J,S)$ whereas the staircase vanishes there. The numerator is unchanged and the denominator strictly decreases. Quantitatively, for $s_J<b_1<b_2<S$ the extra integral is at least
\[
 \delta_J=\min\{c_1,\rho(b_2)\}\big(d(b_2)-d(b_1)\big)>0.
\]
If the staircase has numerator $4A_\Pi$ and denominator $D_\Pi$, its coefficient is improved by at least $4A_\Pi\delta_J/D_\Pi^2$.

\subsection{The relaxed cap class and compact maximizers}\label{sec:caps}
Core replacement can produce a positive measure on an interface circle, so it is useful to allow measure-valued companions from the outset. The resulting class is closed under the compactness argument below and is also the class used in the finite-degree transfer.

For $R\ge R_0$, define
\begin{equation}\label{eq:cap-class}
\begin{split}
\Acal_R=\{p\in H^1(\R^2):\;&p\text{ is real and radial},\quad p\ge0,\quad
 \supp p\subset\overline D_R,\\
 &\int p=1,\quad H=(1+\Delta)p\text{ is a nonnegative finite measure}\}.
\end{split}
\end{equation}
The equation for $H$ is distributional. Set $K=p*H$, $a=K(0)$, and use \eqref{eq:profile-def}. The normalized coefficient is
\begin{equation}\label{eq:cap-C}
 D_p(c)=\frac1{4\pi}-\int_0^{2R}d'(s)\ell_{p,c}(s)\dd s,
 \qquad C(p)=\max_{0\le c\le a}\frac{4(a-c)}{D_p(c)},
 \qquad C_R=\sup_{p\in\Acal_R}C(p).
\end{equation}
When $a=0$, define $C(p)=0$. Multiplying a nonnormalized cap by a positive constant leaves \eqref{eq:planar-functional} unchanged, so it agrees with this normalization.

\begin{lemma}\label{lem:cap-a-priori}
For every $p\in\Acal_R$, the companion has mass one, the cap has a continuous compactly supported representative, and
\begin{align}
 \norm p_\infty&\le U:=\frac{e}{4\pi},&
 \norm{\nabla p}_2^2&\le\norm p_2^2\le U,\label{eq:cap-bounds}\\
 0\le a&=\norm p_2^2-\norm{\nabla p}_2^2\le U,&
 D_p(c)&\ge D_*:=\frac{1-\sqrt3/(2\pi)}{4\pi}>0.\label{eq:cap-aD}
\end{align}
The kernel $K$ is continuous, nonnegative, supported in $\overline D_{2R}$, and has integral one. In radial coordinates,
\begin{equation}\label{eq:radial-flux}
 2\pi r p'(r)=H(D_r)-\int_{D_r}p\dd x\quad\text{for almost every }r>0,
\end{equation}
and hence
\begin{equation}\label{eq:radial-log}
 |rp'(r)|\le\frac1{2\pi},\qquad
 |p(r_1)-p(r_2)|\le\frac1{2\pi}\left|\log\frac{r_1}{r_2}\right|
 \quad(r_1,r_2>0).
\end{equation}
\end{lemma}
\begin{proof}
The mass identity in Subsection~\ref{subsec:radial-preliminaries} gives $H(\R^2)=1$. The heat-kernel majorant \eqref{eq:heat-companion}, with $t=1$, yields $0\le p\le e/(4\pi)=U$ almost everywhere.

By \eqref{eq:planar-companion-identities}, $K$ is continuous and nonnegative, has the asserted support and mass, and satisfies the diagonal energy identity. Nonnegativity of the diagonal gives the gradient estimate, and $\norm p_2^2\le\norm p_\infty\norm p_1\le U$. The denominator bound is \eqref{eq:mass-gap}.

For \eqref{eq:radial-flux}, radial tests on annuli identify the bounded-variation derivative of $rp'$ as the radial measure of $H-p\dd x$. The integration constant at the origin is zero: a nonzero constant would give a $1/r$ term in $p'$, inconsistent with $p\in H^1$. Integration proves the identity almost everywhere; at circle atoms it can also be read with one-sided representatives. The two cumulative masses in \eqref{eq:radial-flux} lie in $[0,1]$, which proves \eqref{eq:radial-log} by integration, including intervals crossing the support boundary.

It remains to specify the value at the origin. The flux identity and $p\le U$ give $p'(r)\ge-Ur/2$. Therefore $p(r)+Ur^2/4$ is nondecreasing and bounded near zero, with a finite limit there. Define $p(0)$ by this limit. On every annulus, radial $H^1$ regularity gives continuity; at the support boundary the zero extension has the same continuous trace. This yields the claimed global representative. The estimates now hold pointwise where appropriate.
\end{proof}

The class is nonempty. The following Bessel profile is the planar baseline used in our cap optimization and is motivated by the small-scale profile associated with Yudin's cap~\cite{Yudin}. Let
\begin{equation}\label{eq:Bessel-cap}
 q(r)=J_0(r)-J_0(R_0)\quad(0\le r\le R_0),\qquad q(r)=0\quad(r>R_0),
 \qquad b_0=-J_0(R_0)>0.
\end{equation}
Since $J_1>0$ before its first positive zero, $q$ is positive and decreasing in the disk, and $q(R_0)=q'(R_0)=0$. The Bessel equation gives $(1+\Delta)q=b_0\ind_{D_{R_0}}$. Also
\begin{equation}\label{eq:Bessel-masses}
 P_0=\int q=\pi R_0^2b_0,\qquad K_q(0)=\pi R_0^2b_0^2.
\end{equation}
Hence $q/P_0\in\Acal_R$ for every $R\ge R_0$, and its zero-height coefficient is $16/R_0^2>0$. In particular $0<C_R<\infty$.

\begin{theorem}\label{thm:cap-compactness}
For every fixed $R\ge R_0$, there is $p_R^*\in\Acal_R$ with $C(p_R^*)=C_R$. Every maximizing sequence has a strongly $H^1(\R^2)$ convergent subsequence.
\end{theorem}
\begin{proof}
The bounds in Lemma~\ref{lem:cap-a-priori} and fixed support give weak $H^1$ and strong $L^2$ compactness. The latter is the usual Rellich argument; alternatively, first approximate uniformly in $L^2$ by convolutions with a mollifier using $\norm{p-p*\eta_\delta}_2\le C\delta\norm{\nabla p}_2$, use compactness for the smooth family on a bounded set, and then let $\delta$ tend to zero. Every resulting limit retains radiality, nonnegativity, support, and mass. Its companion is a positive distribution and therefore a positive measure, so the limit is feasible.

Take a maximizing sequence $p_j$, optimal heights $c_j$, and a subsequence with $p_j\rightharpoonup p$ in $H^1$ and $p_j\to p$ in $L^2$. Put $a_j=K_j(0)$, and pass further so that $a_j\to a_*$ and $c_j\to c$. Weak lower semicontinuity gives only
\begin{equation}\label{eq:diag-defect}
 a:=K(0)\ge a_*;
\end{equation}
the possible strict inequality is the obstruction to obtaining convergence of the diagonal energies from weak convergence alone.

For $s>0$, apply the radial Fourier formula \eqref{eq:radial-Fourier} to $K_j=p_j*H_j$. The large-argument estimate in \eqref{eq:bessel-basic-bounds}, together with the uniform $H^1$ bound, controls the high-frequency tail uniformly for $s\ge\delta>0$. On bounded frequency intervals, fixed support and strong $L^2$ convergence imply uniform convergence of $\widehat p_j$. Thus
\begin{equation}\label{eq:offdiag-convergence}
 K_j\longrightarrow K\quad\text{uniformly on }[\delta,2R]\quad(\delta>0).
\end{equation}

There is also a uniform one-sided estimate near zero. Since $|\widehat p_j|\le1$, $J_0(v)\le1$, and $1-J_0(v)\le v^2/4$, the part $t\ge1$ of $K_j(s)-a_j$ is nonnegative, while the remaining part is at least
\[
 -\frac{s^2}{8\pi}\int_0^1(1-t^2)t^3\dd t=-\frac{s^2}{96\pi}.
\]
Hence
\begin{equation}\label{eq:diagonal-control}
 K_j(s)\ge a_j-\frac{s^2}{96\pi}.
\end{equation}
Since $C(p_j)\to C_R>0$ and $D_{p_j}(c_j)\ge D_*$, there is $\delta_0>0$ such that $a_j-c_j\ge\delta_0$ for all sufficiently large $j$. Choose $\delta>0$ with $\delta^2/(96\pi)<\delta_0/2$. Equation~\eqref{eq:diagonal-control} then places $K_j$ strictly above $c_j$ on $[0,\delta]$, so $\ell_{p_j,c_j}=c_j$ there. The same conclusion holds for the limiting height and kernel because $a\ge a_*$. Away from that interval, \eqref{eq:offdiag-convergence} controls the running minima uniformly. Thus the truncated envelopes converge uniformly on $[0,2R]$, and their denominators converge as well.

It follows that
\[
 C_R=\lim_j C(p_j)=\frac{4(a_*-c)}{D_p(c)}
 \le\frac{4(a-c)}{D_p(c)}\le C(p)\le C_R.
\]
All inequalities are equalities. This proves attainment and also $a=a_*$, since otherwise the first inequality would be strict. The diagonal energy identity and strong $L^2$ convergence now give convergence of the gradient norms. Together with weak convergence this yields strong $H^1$ convergence.
\end{proof}

The compactness statement keeps $R$ fixed. No uniqueness or regularity of the boundary of the maximizing cap is asserted.

\subsection{Positive-potential conjugation and positive companions}\label{sec:transfer}
The transfer must preserve positivity of $(1+\Delta)p$ for every $p\in\Acal_R$, not just for caps with a strictly positive smooth companion density. We choose a radial amplitude that removes the discrepancy between the planar and spherical first-derivative coefficients. The remaining term is nonnegative, and the kernel error can be bounded using only $R$.

Set
\begin{equation}\label{eq:spectral-scale}
 n=2m+2,\qquad k=\sqrt{n(n+1)},\qquad k\ge4R,
\end{equation}
and define
\begin{equation}\label{eq:amplitude}
 J_k(r)=\frac{\sin(r/k)}{r/k},\qquad A_k(r)=J_k(r)^{-1/2},\qquad
 q_n(\cos\theta)=A_k(k\theta)p(k\theta).
\end{equation}
The cap is zero beyond $\theta=R/k$. It is nonnegative, continuous, and belongs to $H^1(\Sph)$.

Using the radial operators in \eqref{eq:radial-laplacians}, logarithmic differentiation gives
\[
 2A_k'/A_k+k^{-1}\cot(r/k)=1/r.
\]
A second differentiation yields the exact identity
\begin{equation}\label{eq:conjugation}
 (1+k^{-2}\Delta_{\Sph})(A_kp)
 =A_k\big((1+\Delta_{\R^2})p+V_kp\big),
\end{equation}
where
\begin{equation}\label{eq:potential}
 V_k(r)=\frac1{4k^2}\left(1+\csc^2(r/k)-\frac{k^2}{r^2}\right),
 \qquad V_k(0)=\frac1{3k^2}.
\end{equation}
For example, putting $b=1/r$ and $c=k^{-1}\cot(r/k)$ gives $(A_k''+cA_k')/A_k=(b'-c')/2+(b^2-c^2)/4$, which is \eqref{eq:potential}. The identity extends from smooth annuli to radial distributions by multiplication with the smooth radial amplitude.

For $0\le\theta\le1/2$, $\sin\theta\le\theta$ and $\sin\theta\ge\theta(1-\theta^2/6)$ imply
\[
 0\le\csc^2\theta-\theta^{-2}
 \le\frac{1/3-\theta^2/36}{(1-\theta^2/6)^2}
 \le\frac{192}{529}<1.
\]
Therefore
\begin{equation}\label{eq:potential-positive}
 0<V_k(r)\le\frac1{2k^2}
\end{equation}
in the required range.

Let $H_n=q_n+k^{-2}\Delta_{\Sph}q_n$, understood as a distribution relative to probability surface measure. Pulling it back to planar normal coordinates and multiplying by $4\pi k^2$ gives the measure identity
\begin{equation}\label{eq:measure-pullback}
 \mu_k=\sqrt{J_k}\,H+\sqrt{J_k}\,V_kp\dd x\ge0.
\end{equation}
In particular, the singular part of $H$ is multiplied by $\sqrt{J_k}>0$, so the measure remains nonnegative.

Define $G_n=q_n*H_n$ by spherical convolution and put
\[
 P_k=4\pi k^2\int_{\Sph}q_n\dd\sigma,
 \qquad \widetilde K_n(s)=4\pi k^2G_n(\cos(s/k)).
\]
The same amplitude also decreases the mass and increases the diagonal kernel value:
\begin{equation}\label{eq:mass-diagonal-benefit}
 P_k=\int\sqrt{J_k}\,p\dd x=\mu_k(\R^2)\le1,
 \qquad \widetilde K_n(0)=a+\int V_kp^2\dd x\ge a.
\end{equation}
The weak identity $\int p\dd H=\norm p_2^2-\norm{\nabla p}_2^2$ can be justified by mollification converging uniformly and strongly in $H^1$; compact continuity of $p$ controls the measure pairing.

By \eqref{eq:companion-spectrum}, if $q_n=\sum a_\ell P_\ell$, then $G_n$ has Legendre coefficients
\begin{equation}\label{eq:conjugated-spectrum}
 \left(1-\frac{\ell(\ell+1)}{n(n+1)}\right)
 \frac{|a_\ell|^2}{2\ell+1}.
\end{equation}
They are nonpositive for $\ell\ge n$. The absolute uniform convergence and nonnegativity follow from Subsection~\ref{subsec:lp-turan-preliminaries}, since $q_n\in H^1(\Sph)$ and $q_n,H_n\ge0$. Thus $G_n$ has the spectral and positivity properties required for degree $n-1$.

\subsection{Uniform kernel and packing errors}
On the radius-$2R$ scaled normal-coordinate disk, the spherical metric is
\[
 dr^2+k^2\sin^2(r/k)\dd\phi^2.
\]
Its length factors lie between $1-2R^2/(3k^2)$ and one. The corresponding spherical disk has angular radius at most $1/2$ and is geodesically convex. Comparing a planar segment with a spherical minimizing geodesic gives, for $|u|\le R$ and $0\le s\le2R$,
\begin{equation}\label{eq:distance-compare}
 (1-2R^2/(3k^2))|u-se_1|\le d_k(u,se_1)\le|u-se_1|,
\end{equation}
where $d_k$ is the scaled spherical distance. Put $t=R^2/k^2\le1/16$. Since
\[
 -\log(1-2t/3)\le\frac{2t/3}{1-2t/3}\le\frac{16}{23}t<t,
\]
the logarithmic control \eqref{eq:radial-log} yields
\begin{equation}\label{eq:distance-function}
 |p(d_k(u,se_1))-p(|u-se_1|)|\le\frac{R^2}{2\pi k^2}.
\end{equation}
At coincident arguments both values agree. Thus the relative metric comparison is controlled by the logarithmic modulus in \eqref{eq:radial-log}.

Using $1-J_k\le r^2/(6k^2)$ and $(1-t)^{-1/2}\le1+t$ for $0\le t\le1/2$ gives
\[
 \norm{A_kp-p}_\infty\le\frac{UR^2}{6k^2},\qquad
 \norm{\mu_k-H}_{\mathrm{TV}}\le\frac{R^2/6+1/2}{k^2}.
\]
In the first inequality it is enough to estimate where $p\ne0$. Write
\[
 \widetilde K_n(s)=\int A_k(d_k(u,se_1))p(d_k(u,se_1))\dd\mu_k(u)
\]
and compare with $K(s)=\int p(|u-se_1|)\dd H(u)$. Using $\mu_k(\R^2)\le1$ and \eqref{eq:distance-function} yields
\begin{equation}\label{eq:uniform-kernel-error}
 \norm{\widetilde K_n-K}_{C[0,2R]}\le\frac{E_R}{k^2},
 \qquad E_R=U\left(\frac{R^2}{3}+\frac12\right)+\frac{R^2}{2\pi}.
\end{equation}

For the packing function from Subsection~\ref{subsec:lp-turan-preliminaries}, \eqref{eq:packing-error} gives, with $\beta_n(s)=k^2f(s/k)$,
\begin{equation}\label{eq:scaled-packing-error}
 |\beta_n(s)-d(s)|\le\frac{F_R}{k^2},\qquad F_R=\frac23R^4,
 \qquad 0\le s\le2R.
\end{equation}

\subsection{Transfer of all heights and all caps}
\begin{theorem}\label{thm:effective-transfer}
Define
\begin{equation}\label{eq:MR}
 B_R=E_R\frac{\sqrt3R^2}{2\pi}+\frac23UR^4,
 \qquad M_R=\frac{UB_R}{D_*^2}.
\end{equation}
For $n=2m+2$ and $n(n+1)\ge16R^2$,
\begin{equation}\label{eq:effective-MR}
 N_{2m+1}\ge C_R(m+1)(m+3/2)-M_R.
\end{equation}
For every $R\ge R_0$ and $m\ge3$, without a spectral threshold,
\begin{equation}\label{eq:effective-19}
 N_{2m+1}\ge\left\lceil C_R(m+1)(m+3/2)-19R^4\right\rceil.
\end{equation}
\end{theorem}
\begin{proof}
Fix $p\in\Acal_R$ and $0\le c<a$, and put $A=a-c$, $D=D_p(c)$. Lower and truncate the envelope:
\[
 \ell_n=(\ell_{p,c}-E_R/k^2)_+.
\]
It is a continuous nonincreasing lower bound for $\widetilde K_n$. By the favorable diagonal inequality in \eqref{eq:mass-diagonal-benefit},
\[
 \widetilde K_n(0)-\ell_n(0)\ge a-c=A.
\]
This remains valid when $c\le E_R/k^2$, including $c=0$, so there is no height-dependent threshold.

Applying Proposition~\ref{prop:continuous-sphere} at the physical scale gives
\begin{equation}\label{eq:transfer-ratio}
 N_{n-1}\ge k^2
 \frac{\widetilde K_n(0)-\ell_n(0)}
 {P_k^2/(4\pi)-\int_0^{2R}\beta_n(s)(-d\ell_n(s))}.
\end{equation}
Its actual denominator is positive by \eqref{eq:packing-derivative}. Since the Stieltjes measure has mass at most $c\le U$, equations~\eqref{eq:uniform-kernel-error} and \eqref{eq:scaled-packing-error} give
\begin{align*}
 \int_0^{2R}\beta_n(-d\ell_n)
 &\ge\int_0^{2R}d'(s)\ell_n(s)\dd s-\frac{F_RU}{k^2}\\
 &\ge\int_0^{2R}d'(s)\ell_{p,c}(s)\dd s
 -\frac{E_Rd(2R)+F_RU}{k^2}.
\end{align*}
Using $P_k\le1$ in \eqref{eq:transfer-ratio} therefore yields
\[
 N_{n-1}\ge k^2\frac{A}{D+B_R/k^2}
 \ge k^2\frac AD-\frac{AB_R}{D^2}
 \ge k^2\frac AD-M_R.
\]
The penultimate inequality uses $(1+u)^{-1}\ge1-u$. Take the best height and then the supremum over all caps, using the error uniformity. Since $k^2/4=(m+1)(m+3/2)$, this proves \eqref{eq:effective-MR}. The case $a=0$ is trivial and does not affect the supremum. Attainment or regularity of an optimizing cap is not required for this step.

To make the constant elementary, use $e<11/4$, $3<\pi<22/7$, $\sqrt3<7/4$, and $R\ge1$. Then
\[
 U<\frac{11}{48},\qquad D_*>\frac{119}{2112},\qquad
 d(2R)<\frac7{24}R^2,
\]
so
\[
 E_R\le\frac{103}{288}R^2,\qquad
 B_R\le\frac{1777}{6912}R^4,\qquad
 M_R<\frac{2365187}{127449}R^4<19R^4.
\]
This proves \eqref{eq:effective-19} when $k\ge4R$.

For each fixed $R$, \eqref{eq:effective-MR} and Corollary~\ref{cor:upper}, upon division by $m^2$ and passage to infinity, already give $C_R\le13/8$. If $k<4R$, then
\[
 C_R(m+1)(m+3/2)=C_Rk^2/4<\frac{13}{2}R^2<19R^4.
\]
The right-hand side of \eqref{eq:effective-19} is nonpositive, so the bound is automatic. This covers all $m\ge3$.
\end{proof}

The constant $M_R$ is independent of the cap, its companion measure, and the chosen height. Hence the same estimate applies to the maximizing cap of Theorem~\ref{thm:cap-compactness}. Above the spectral threshold, $M_R$ gives the sharper finite-degree remainder; $19R^4$ is the form valid for all $m\ge3$.

\section{Cap variations and radius optimization}\label{sec:variations}
We now vary the cap itself. The first variation below increases the coefficient of the baseline Bessel cap. A finite replacement on a central disk then maximizes the diagonal energy at fixed core mass and boundary trace. To turn this energy gain into a gain in the full quotient, the replacement must leave the truncated kernel envelope unchanged; this is the purpose of the protection conditions.

\subsection{The baseline Bessel cap and a strict first variation}
For the cap $q$ in \eqref{eq:Bessel-cap}, let
\[
 K_0=b_0(q*\ind_{D_{R_0}}),\qquad
 \eta(a)=K_0(R_0a)/K_0(0),\quad0\le a\le2.
\]
The function $\eta$ is strictly decreasing before its support endpoint. To see this, express the strictly decreasing radial function $q$ as a positive integral of disk indicators; the overlap of two disks decreases with separation, and for each $0<s<2R_0$ a positive-measure family of those overlaps decreases strictly. Also $\eta(0)=1$, $\eta(2)=0$, and \eqref{eq:Bessel-masses} gives
\[
 \int_0^2 a\eta(a)\dd a=\frac12.
\]
Write
\[
 \beta(a)=\frac{\sqrt3a^2}{2\pi},\quad
 I(a)=\int_a^2\beta(u)(-d\eta(u)),\quad
 \Phi(a)=1-I(a)-\beta(a)(1-\eta(a)).
\]
Integration by parts gives
\[
 \Phi(a)=1-\beta(a)-\int_a^2\beta'(u)\eta(u)\dd u,
 \qquad \Phi'(a)=-\beta'(a)(1-\eta(a))<0.
\]
Its endpoint values are $1-\sqrt3/(2\pi)>0$ and $1-2\sqrt3/\pi<0$. Thus there is a unique $a_*\in(0,2)$ with $\Phi(a_*)=0$. Proposition~\ref{prop:height-root} gives
\begin{equation}\label{eq:bessel-Cstar}
 C(q)=C_*=
 \frac{16(1-\eta(a_*))}{R_0^2(1-I(a_*))}
 =\frac{32\pi}{\sqrt3R_0^2a_*^2}>\frac{16}{R_0^2}.
\end{equation}
The strict inequality also follows directly from the positive tail term in \eqref{eq:height-root}. At separation $\sqrt2R_0$, the overlap lens lies strictly inside a half-disk; hence $\eta(\sqrt2)<1/2$. It follows that
\[
 \Phi(\sqrt2)>1-\frac{3\sqrt3}{2\pi}>0,\qquad a_*>\sqrt2.
\]

The original cap is not stationary in the full feasible shape class. Set
\[
 v(r)=q(r)\frac{r^2}{R_0^2},\qquad p_\eps=q+\eps v.
\]
Its companion is $b_0\ind_{D_{R_0}}+\eps w$, where, within the disk,
\[
 w(r)=\frac{b_0r^2-4rJ_1(r)+4q(r)}{R_0^2}.
\]
For sufficiently small positive $\eps$, both $p_\eps$ and its companion are positive in the disk, and the boundary double zero is unchanged.

To compute the variation of the normalized kernel, set $g(r)=r^2/R_0^2$, let
\[
 d\nu(r)=\frac{rq(r)\dd r}{\int_0^{R_0}rq(r)\dd r},
\]
and let $\omega_s(r)$ be the angular length of the circle of radius $r$ lying in the disk centered at $se_1$ with radius $R_0$. Commutation of convolution with $1+\Delta$ gives $\dot K=2b_0v*\ind_{D_{R_0}}$. Therefore, for the normalized kernel profile at physical separation $s$,
\begin{equation}\label{eq:covariance-variation}
 \left.\frac d{d\eps}\frac{K_\eps(s)}{K_\eps(0)}\right|_{\eps=0}
 =\frac1\pi\Cov_\nu(g,\omega_s).
\end{equation}
For $s\ge\sqrt2R_0$, the nonzero overlap angle is increasing in $r$ on a positive-measure interval. This follows by differentiating $(r^2+s^2-R_0^2)/(2rs)$, the cosine of its half-angle. Since $g$ is strictly increasing, the covariance is positive for $\sqrt2R_0\le s<2R_0$.

Put $s_*=R_0a_*$. The derivative of $K_\eps(0)/P_\eps^2$ is zero at the base cap, because $K_0(0)=b_0P_0$ and $\dot K(0)=2b_0\dot P$. At the optimal base height, the derivative contributions at $s_*$ cancel by the root equation. In terms of $\widehat\eta_\eps(s)=K_\eps(s)/K_\eps(0)$, the logarithmic first variation of the fixed-threshold quotient is
\begin{equation}\label{eq:first-shape-gain}
 \frac{\displaystyle\int_{s_*}^{2R_0}d'(s)\dot{\widehat\eta}_0(s)\dd s}
 {d(s_*)(1-\widehat\eta_0(s_*))}>0.
\end{equation}
This calculation is compatible with the running minimum: the base kernel is strictly decreasing, so its minimizing point in each prefix interval is unique. The first variation of that minimum is the kernel variation at the endpoint; uniform bounds on the difference quotients justify integration. Thus a small feasible perturbation, followed by height optimization, has coefficient strictly greater than $C_*$. After normalization it belongs to $\Acal_{R_0}$, so in particular $C_{R_0}>C_*>16/R_0^2$. The strict gain is qualitative; the argument does not evaluate $C_{R_0}$.

\subsection{A finite, mass-preserving core replacement}
Let $p$ be radial and $C^4$ on $\overline D_R$, positive in $D_R$, with $p(R)=p'(R)=0$ and $p'(0)=0$. Suppose its ordinary companion density
\[
 h=(1+\Delta)p\ge\eta>0\quad\text{on }D_R
\]
is extended by zero outside. Let $P=\int p=\int h$, $K=p*h$, and choose $0<c<K(0)$. Use the nonnormalized denominator $D(c)$ from \eqref{eq:planar-functional}.

Choose $0<b<R$ with $b\le1$, assume $h$ is nondecreasing and nonconstant on $[0,b]$, and define
\begin{equation}\label{eq:core-conditions-def}
 V_b=\osc_{[0,b]}h,\quad U_b=\frac{b^2V_b}{4-b^2},\quad
 p_b^{\min}=\min_{[0,b]}p,\quad
 \mu_b=\min_{0\le s\le R+b}(K(s)-c).
\end{equation}
The required protection conditions are
\begin{equation}\label{eq:core-protection}
 U_b<p_b^{\min},\qquad
 \mu_b>2U_bP+2\pi b^2U_bV_b.
\end{equation}
The first inequality preserves cap positivity; the second protects the truncated envelope. Both are additional assumptions on the chosen radius and height.

Define
\begin{equation}\label{eq:core-replacement}
\begin{split}
 g_b(r)&=\frac{J_0(r)}{J_0(b)}-1,\qquad
 \lambda_b=\frac{\int_{D_b}g_bh\dd x}{\int_{D_b}g_b\dd x},\\
 p_b^\sharp(r)&=\lambda_b+(p(b)-\lambda_b)\frac{J_0(r)}{J_0(b)}
 \quad(0\le r\le b).
\end{split}
\end{equation}
Let $\widetilde p=p_b^\sharp$ in $D_b$ and $\widetilde p=p$ outside $D_b$.

\begin{theorem}\label{thm:core-replacement}
Under these conditions, $\widetilde p\ge0$, its mass is $P$, and its companion is the positive measure
\begin{equation}\label{eq:positive-interface}
 (1+\Delta)\widetilde p=
 \lambda_b\ind_{D_b}\dd x+h\ind_{D_R\setminus D_b}\dd x
 +j_b\mathcal H^1|_{\partial D_b},\qquad
 j_b=p'(b)-(p_b^\sharp)'(b)>0.
\end{equation}
If $u=\widetilde p-p$ and
\[
 \mathcal E_b=\int_{D_b}(|\nabla u|^2-u^2)\dd x,
\]
then $\mathcal E_b>0$, $\widetilde K(0)=K(0)+\mathcal E_b$, and the full envelope at height $c$ is unchanged. Consequently the fixed-height coefficient increases by exactly $4\mathcal E_b/D(c)$. The replacement uniquely maximizes the diagonal energy among all core perturbations with the same boundary trace and mass.
\end{theorem}
\begin{proof}
For $b\le1$, $g_b>0$ in the disk and $(-\Delta-1)g_b=1$, with zero boundary trace. The Bessel equation gives
\begin{equation}\label{eq:core-u}
 (-\Delta-1)u=h-\lambda_b\quad\text{in }D_b,\qquad
 u|_{\partial D_b}=0.
\end{equation}
Pairing with $g_b$ shows $\int_{D_b}u=\int_{D_b}g_b(h-\lambda_b)=0$.

For $v\in H_0^1(D_b)$, the radial Cauchy--Schwarz estimate
\[
 |v(r,\theta)|^2\le\log(b/r)\int_r^b s|\partial_rv(s,\theta)|^2\dd s
\]
and integration yield
\begin{equation}\label{eq:core-poincare}
 \int_{D_b}v^2\le\frac{b^2}{4}\int_{D_b}|\nabla v|^2.
\end{equation}
Thus $-\Delta-1$ has a positive Dirichlet quadratic form and obeys the comparison principle on this disk. Since $(b^2-r^2)/(4-b^2)$ is a supersolution for the constant right side one, $|h-\lambda_b|\le V_b$ gives $\norm u_\infty\le U_b$. The first inequality in \eqref{eq:core-protection} proves positivity of the new cap.

Integrate \eqref{eq:core-u} and use its zero mass:
\begin{equation}\label{eq:jump-mass}
 2\pi b j_b=\int_{D_b}h\dd x-\pi b^2\lambda_b.
\end{equation}
The function $g_b$ is decreasing and $h$ is nondecreasing, nonconstant. The reverse Chebyshev inequality, obtained by integrating $(g_b(r)-g_b(s))(h(r)-h(s))\le0$, is strict and shows that the weighted mean $\lambda_b$ is smaller than the ordinary disk mean of $h$. Hence $j_b>0$. The jump in the radial derivative contributes the positive circle measure in \eqref{eq:positive-interface}.

Set $\mathcal B(v,w)=\int_{D_b}(\nabla v\cdot\nabla w-vw)\dd x$. This is positive definite by \eqref{eq:core-poincare}. Equation~\eqref{eq:core-u} and the mass constraint give
\[
 \mathcal E_b=\mathcal B(u,u)=\int_{D_b}hu\dd x>0;
\]
otherwise $u=0$, forcing $h=\lambda_b$, contrary to nonconstancy. Expanding the whole-cap energy yields
\[
 \widetilde K(0)-K(0)=2\int hu-\mathcal B(u,u)=\mathcal E_b.
\]
More generally, for every $v\in H_0^1(D_b)$ with zero mass,
\begin{equation}\label{eq:core-maximality}
 \left(\int(p+v)^2-\int|\nabla(p+v)|^2\right)-K(0)
 =\mathcal E_b-\mathcal B(v-u,v-u).
\end{equation}
This proves uniqueness of the maximizing perturbation, even without a radiality restriction on $v$.

Finally put $\delta H=(1+\Delta)u$. Its total variation is at most $2\pi b^2V_b$, by \eqref{eq:positive-interface} and \eqref{eq:jump-mass}. Distributional convolution gives
\[
 \widetilde K-K=2u*H+u*\delta H,
\]
supported in $\overline D_{R+b}$, with uniform norm at most $2U_bP+2\pi b^2U_bV_b<\mu_b$. Thus $\widetilde K>c$ on $[0,R+b]$ and $\widetilde K=K$ thereafter. Their running-minimum envelopes truncated at $c$ are identical. Mass and denominator are unchanged, while the diagonal gains $\mathcal E_b$, proving the coefficient assertion.
\end{proof}

If $c$ was an optimal original height, the new optimized coefficient is strictly greater than $C(p)$. There is also an explicit quantitative gain. If $h'(r)\ge\kappa r$ on $[0,b]$, take
\[
 \varphi(r)=(b^2-r^2)(r^2-b^2/3).
\]
Its mass is zero, and direct integration gives
\[
 \int r^2\varphi=\frac{\pi b^8}{36},\qquad
 \mathcal B(\varphi,\varphi)=\frac{2\pi b^8}{135}(30-b^2).
\]
The function $h-\kappa r^2/2$ is nondecreasing and $\varphi$ has one sign change from negative to positive, so $\int h\varphi\ge\pi\kappa b^8/72$. Cauchy--Schwarz in $\mathcal B$ then yields
\begin{equation}\label{eq:core-quantitative}
 \mathcal E_b\ge\frac{5\pi\kappa^2b^8}{384(30-b^2)},\qquad
 \frac{4\mathcal E_b}{D(c)}\ge\frac{5\pi\kappa^2b^8}{96(30-b^2)D(c)}.
\end{equation}
After normalization, the new cap belongs to $\Acal_R$; Theorem~\ref{thm:effective-transfer} applies directly to its positive interface measure.

\subsection{Moving the protected core radius}
The same construction is defined for $0<s<\min\{2,R\}$. On $[0,2)$, the signs $J_0>0$ and $J_1>0$ follow from their alternating series. Replace $b$ by $s$ in \eqref{eq:core-replacement}, and write the resulting perturbation as $u_s$, its jump as $j_s=-u_s'(s^-)$, and its energy as $\mathcal E(s)$. Then
\begin{equation}\label{eq:shape-radius-derivative}
 \mathcal E'(s)=2\pi s j_s^2.
\end{equation}
To prove this, differentiate $(-\Delta-1)u_s=h-\lambda_s$, $u_s|_{\partial D_s}=0$, and $\int u_s=0$. The parameter derivative has boundary value $j_s$, zero integral, and satisfies $(-\Delta-1)\dot u_s=-\lambda_s'$. Hence
\[
 \mathcal B(\dot u_s,u_s)=0,\qquad
 \mathcal B(u_s,\dot u_s)=\int h\dot u_s-2\pi s j_s^2.
\]
Also $\mathcal E(s)=\int_{D_s}hu_s$, whose derivative has no domain-boundary term because $u_s=0$ there. Symmetry of $\mathcal B$ proves \eqref{eq:shape-radius-derivative}.

At a radius $b$ satisfying the strict conditions of Theorem~\ref{thm:core-replacement}, $j_b>0$. Continuity of the explicit formulas and of the protection margins allows a radius $B>b$, $B<\min\{2,R\}$, such that for all $s\in[b,B]$,
\[
 j_s\ge j_b/2,\qquad U_s<p_s^{\min},\qquad
 \mu_s>2U_sP+2\pi s^2U_sV_s.
\]
The new companion remains positive because its density $\lambda_s$ lies between the minimum and maximum of $h$, and its jump is nonnegative. No extension of monotonicity of $h$ to all of $[0,B]$ is required. The mass and protected envelope remain fixed. Thus, for the fixed-height coefficients $C_s=4(K(0)+\mathcal E(s)-c)/D(c)$,
\begin{equation}\label{eq:radius-gain}
 C_B-C_b=\frac{8\pi}{D(c)}\int_b^B s j_s^2\dd s
 \ge\frac{\pi j_b^2(B^2-b^2)}{D(c)}>0.
\end{equation}
The continuation is restricted to the interval on which the displayed positivity and protection inequalities hold.

\subsection{A nonempty family of protected replacements}
The assumptions in \eqref{eq:core-protection} can be met explicitly. With $q$ from \eqref{eq:Bessel-cap}, set
\begin{equation}\label{eq:quartic-family}
 p_\eps(r)=q(r)\left(1+\eps\frac{r^4}{R_0^4}\right),\qquad0\le r\le R_0,
\end{equation}
for a sufficiently small fixed $\eps>0$. Its companion is $h_\eps=b_0+\eps w$, where
\[
 w(r)=\frac{b_0r^4-8r^3J_1(r)+16r^2q(r)}{R_0^4},\qquad
 w'(r)=\frac{32q(0)}{R_0^4}r+O(r^3).
\]
Thus $h_\eps\ge b_0/2$ for small $\eps$, and on a sufficiently small core,
\[
 h_\eps'(r)\ge\frac{16\eps q(0)}{R_0^4}r.
\]
At $\eps=0$, the kernel is strictly decreasing and its optimal root $s_0$ is also its first contact radius. By \eqref{eq:root-value}, the effective lower bound, and the constructive coefficient $13/8$,
\[
 s_0^2\ge\frac{256\pi}{13\sqrt3}>25.
\]
Since $R_0<4$, this gives $s_0>R_0+1$. Under a small perturbation, kernels and running minima converge uniformly, and the uniquely isolated root and its height converge as well. Hence the kernel on $[0,R_0+1/2]$ stays strictly above the optimal height. Taking $b$ still smaller gives $V_b=O(b^2)$, $U_b=O(b^4)$, while the required positive margins have positive limits. This proves \eqref{eq:core-protection} for a nonempty family of caps.

A related local necessary condition explains the Bessel form of the replacement. Suppose the original height is protected beyond $R+a$, and allow sufficiently small smooth radial variations supported in $D_a$ with zero mass. Such a variation changes the kernel only before the protected threshold, leaving the denominator unchanged. Its diagonal first variation is $2\int h\varphi$. Local stationarity for all these variations forces $h$ to be constant in the core. The radial equation then gives $p(r)=\lambda+A J_0(r)$. This is a local necessary condition, not a classification of the relaxed maximizers.

\subsection{Optimization over the support radius}\label{sec:consequences}
Proposition~\ref{prop:yudin}, Theorem~\ref{thm:effective-transfer}, and Corollary~\ref{cor:upper} give the finite bounds in Theorem~\ref{thm:uniform-main}; Theorem~\ref{thm:cap-compactness} gives attainment. After division by $m^2$, these bounds yield its asymptotic interval. The remaining fixed-radius estimate follows by comparing the finite upper and lower bounds at a degree chosen in terms of $R$.

\begin{proposition}\label{prop:fixed-radius-gap}
For each $R\ge R_0$,
\begin{equation}\label{eq:fixed-radius-gap}
 C_R\le\frac{13}{8}-\frac1{16R^4}.
\end{equation}
Furthermore,
\begin{equation}\label{eq:radius-optimization}
 N_{2m+1}\ge
 \left\lceil\sup_{R\ge R_0}\left\{C_R(m+1)(m+3/2)-19R^4\right\}\right\rceil
 \qquad(m\ge3).
\end{equation}
\end{proposition}
\begin{proof}
For even $m\ge64$, the second branch in \eqref{eq:envelope-def}, $\gamma<1$, $m^{2/3}\le m/4$, and $10/3\le m/4$ give
\[
 N_{2m+1}\le\frac{13}{8}m^2+\frac32m.
\]
Choose $m=2\lceil8R^4\rceil$ and write $\delta=13/8-C_R$. Since $R_0>2$ (the alternating series for $J_1$ is positive on $(0,2]$), this $m$ is at least $64$. Comparing with \eqref{eq:effective-19} gives
\[
 \delta(m+1)(m+3/2)\ge\frac{41}{16}m+\frac{39}{16}-19R^4
 \ge\frac{11}{8}m.
\]
Using $(m+1)(m+3/2)\le m(m+3)$ and $m<16R^4+2$ yields
\[
 \delta\ge\frac{11}{8(m+3)}\ge\frac{11}{8(16R^4+5)}\ge\frac1{16R^4}.
\]
This proves \eqref{eq:fixed-radius-gap}. Since the effective bound holds for every $R$, its supremum gives \eqref{eq:radius-optimization}. The classes are nested, so $C_R$ is nondecreasing; it is bounded by $13/8$, and the negative quartic term ensures finiteness of the supremum.
\end{proof}

For a given degree, the lower bound is the maximum of the Fisher gap, the classical root bound, and the radius-optimized cap bound. Above the spectral threshold one may retain $M_R$ instead of $19R^4$. On the upper side, the exact integer budget \eqref{eq:exact-budget} can be smaller than its closed envelope, and the specialized low-degree formulas can be used separately.

Theorem~\ref{thm:sharp-remainder} limits further improvement by optimizing the same scalar budget. A smaller constructive quadratic coefficient would require a different count or completion mechanism. For the lower bound, the maximizers at fixed radius are not explicitly known from the present argument, and their uniqueness and boundary regularity remain undetermined here. The analysis also gives no effective convergence rate for $C_R$ as $R\to\infty$ and no maximizer over all support radii. It therefore does not establish convergence of $N_{2m+1}/m^2$ or equality of the two quadratic coefficients.

We conclude with the following conjecture.

\begin{conjecture}[Asymptotic node count]\label{conj:asymptotic-node-count}
Let $N_t$ denote the minimum number of nodes in a positive cubature formula of degree $t$ on $S^2$. The limit
\[
 L:=\lim_{m\to\infty}\frac{N_{2m+1}}{m^2}
\]
exists. Equivalently,
\[
 \liminf_{m\to\infty}\frac{N_{2m+1}}{m^2}
 =\limsup_{m\to\infty}\frac{N_{2m+1}}{m^2}.
\]
\end{conjecture}
If the limit exists, determine its exact value.

\section*{AI statement}
Artificial-intelligence-assisted tools were used in the preparation of this manuscript for language polishing, organization, exploratory checking, and assistance with routine derivations. All central mathematical ideas, the formulation of the main results, and all major proof decisions were human-generated; the authors retained control of the proof strategy and mathematical judgment throughout. As an additional verification layer, the complete proof was checked using the \texttt{qmd-prover} skill, developed by Xiao Ma; for details, see the GitHub project at \url{https://github.com/powergiant/qmd-prover}. Independently, the human authors also verified the correctness of the arguments and conclusions. The authors take full responsibility for the content and mathematical correctness of the manuscript.

\raggedbottom

\clearpage
\appendix
\section{Explicit mixed kernels in degrees seven and nine}\label{app:kernels}
For $m=3,4$, the mixed kernels in Lemma~\ref{lem:mixed-zero} have the following explicit bases. Their coefficients and pairings are checked using the moment formula \eqref{eq:spherical-moments}.

For $m=3$, a basis of $W_3$ is
\begin{align*}
 q_1&=X(X^2-3Y^2),&q_2&=X(X^2-Z^2),\\
 q_3&=Y(3X^2-Y^2),&q_4&=Y(3X^2-Z^2).
\end{align*}
The pairing $\int Zpq$ with $p\in H_2$ has rank six by positivity on $ZH_2$, hence kernel dimension four. Parity makes all but one monomial test vanish for each of the first two forms and all but one for each of the last two. The remaining tests are $p=XZ$ for $q_1,q_2$ and $p=YZ$ for $q_3,q_4$, where \eqref{eq:spherical-moments} gives zero. Linear independence is immediate from the coefficients of $X^3,XY^2,XZ^2,Y^3,YZ^2$. If $XY\ne0$, the equations $q_1=q_3=0$ require both $X^2=3Y^2$ and $Y^2=3X^2$, impossible. If either coordinate vanishes, the other also vanishes. Thus the only common projective zero is $[0:0:1]$.

For $m=4$, a basis of $W_4$ is
\begin{align*}
 q_1&=XY(X^2-Y^2),\\
 q_2&=XY(X^2-Z^2),\\
 q_3&=X^4-6X^2Y^2+Y^4,\\
 q_4&=X^4-3X^2Y^2-X^2Z^2+Y^2Z^2,\\
 q_5&=5X^4-10X^2Z^2+Z^4.
\end{align*}
For $q_1,q_2$, only the cubic test $p=XYZ$ survives parity in $\int Zpq$. For $q_3,q_4,q_5$, only $p=X^2Z,Y^2Z,Z^3$ need be checked. Substitution of \eqref{eq:spherical-moments} makes each of these pairings zero. The forms are linearly independent, and the dimension is five by Lemma~\ref{lem:mixed-zero}. If $XY=0$, $q_3=0$ forces $X=Y=0$, after which $q_5=0$ forces $Z=0$. If $XY\ne0$, the first two equations give $X^2=Y^2=Z^2$, but then $q_3=-4X^4\ne0$. Thus the common real projective zero set is empty.

\Needspace{15\baselineskip}
The same trace calculation gives a dimension-dependent obstruction in degree seven. For degree seven on $S^{d-1}$, put $D=\binom{d+2}{3}$. Equality $N=2D$ forces equally weighted antipodal directions with inner products $0$ or $\pm\sqrt{3/(d+4)}$, and $\sum_jv_jv_j^T=(D/d)I_d$. If $e$ counts nonorthogonal unordered pairs, the trace equation
\[
 D+\frac{6e}{d+4}=\frac{D^2}{d}
\]
gives
\[
 e=\frac{d(d-1)(d+1)(d+2)(d+4)^2}{216}.
\]
Nonintegrality excludes equality. For $d=3$ this is the $245/9$ obstruction in the main text.

\clearpage
\section{A Lobatto-latitude form of the Reznick formula}\label{app:latitude22}
Let $\tau(x_1,x_2,x_3)=(x_3,x_2,x_1)$. Applying this orthogonal coordinate permutation to the Reznick formula in Proposition~\ref{prop:22} gives the same twenty-two-node cubature rule in the five-latitude form below. Thus this is a coordinate presentation of the cited Reznick rule, rather than a second inequivalent rule; it is included to make explicit the relation with the Lobatto--Toeplitz construction of Section~\ref{sec:cascade}. Set
\[
 z_* =\sqrt{3/7},\qquad r_*=2/\sqrt7,\qquad \alpha=\arccos(1/7).
\]
Since $\cos(\alpha/2)=2/\sqrt7$ and $\sin(\alpha/2)=\sqrt{3/7}$, the nodes below are exactly the $\tau$-images of the directions in \eqref{eq:22directions}. Their individual sphere weights are as follows:
\begin{center}
\begin{tabular}{@{}p{0.68\textwidth}rr@{}}
\toprule
Nodes & Number & Weight\\
\midrule
$(0,0,\pm1)$ &2&$1/20$\\[3pt]
$(r_*\cos(k\pi/3),r_*\sin(k\pi/3),\pm z_*)$, $0\le k\le5$ &12&$49/1080$\\[3pt]
$(\pm1,0,0)$ &2&$1/27$\\[3pt]
$(0,\pm1,0)$ &2&$1/20$\\[3pt]
$(\cos(\alpha/2),\pm\sin(\alpha/2),0)$ and their antipodes &4&$49/1080$\\
\bottomrule
\end{tabular}
\end{center}
The signs at nonzero latitudes are independent of $k$. The five latitude masses are $1/20$, $49/180$, $16/45$, $49/180$, $1/20$, the degree-seven Lobatto weights at $1,z_*,0,-z_*,-1$.

The equatorial moments admit a direct verification. At even frequency $2j$, its unnormalized moment is
\[
 M_{2j}=\frac2{27}+\frac{(-1)^j}{10}
 +\frac{49}{270}\cos(j\alpha).
\]
Since $\cos\alpha=1/7$, $\cos2\alpha=-47/49$, and $\cos3\alpha=-143/343$, one obtains
\[
 M_2=M_4=0,\qquad M_6=-\frac{32}{315}
 =\frac{16}{45}\left(-\frac27\right).
\]
All odd equatorial moments vanish by antipodality. The two latitude hexagons eliminate all nonzero frequencies through six except frequency six, whose total spherical contribution is
\[
 2\frac{49}{180}(1-3/7)^3=\frac{32}{315}.
\]
This cancels the equatorial term. The zero frequency is handled by the Lobatto latitude rule, and all odd polynomials vanish by antipodality. Hence this is a positive degree-seven formula with twenty-two distinct nodes. This direct verification exhibits the Lobatto--Toeplitz structure of the cited Reznick rule.

\end{document}